\documentclass[11pt]{amsart}
\usepackage[utf8]{inputenc}

\title{Finiteness properties of asymptotically rigid handlebody groups}
\author{Sergio Domingo-Zubiaga}
\date{April 2025}

\usepackage{enumitem}
\usepackage[export]{adjustbox}
\usepackage{array}
\usepackage[utf8]{inputenc}
\usepackage{amsfonts, amssymb, amsmath, amsthm}
\usepackage{graphicx}
\usepackage{parskip}
\usepackage{color}
\usepackage{float}
\usepackage{dirtytalk}
\usepackage{mathtools}
\usepackage{tikz-cd}
\usepackage{thm-restate}
\usepackage{ amssymb }
\usepackage{hyperref}
\usepackage{caption}

\begin{document}

\newtheorem{theor}{Theorem}[section]
\newtheorem{lem}[theor]{Lemma}
\newtheorem{cla}[theor]{Claim}
\newtheorem{sublem}[theor]{Sublemma}
\newtheorem{cor}[theor]{Corollary}
\newtheorem{prop}[theor]{Proposition}

\newtheorem{theorem}{Theorem}
\renewcommand*{\thetheorem}{\Roman{theorem}}
\newtheorem{corollary}[theorem]{Corollary}

\theoremstyle{remark}
\newtheorem{ex}[theor]{Example}
\newtheorem{rem}[theor]{Remark}
\theoremstyle{definition}
\newtheorem{defi}[theor]{Definition}

\def\Map{\textnormal{Map}}
\def\QA{\textnormal{QAut}}
\def\Aut{\textnormal{Aut}}
\def\So{\textnormal{Sym}_o}
\def\AI{\textnormal{AIsom}}
\def\capp{\textnormal{cap}}
\def\push{\textnormal{push}}
\def\forget{\textnormal{forget}}
\def\twist{\textnormal{twist}}
\def\X{\chi}
\def\Xc{\mathfrak{X}}
\def\R{\mathfrak{R}}
\def\P{\mathcal P}
\def\Pc{\mathcal P_d}
\def\B{\mathcal B}
\def\C{\mathcal C}
\def\T{\mathcal T}
\def\D{\mathcal D}
\def\K{\mathcal K}
\def\A{\mathcal A}

\def\dl{\textnormal{lk}^{\downarrow}}

\def\b{\beta}
\def\a{\alpha}

\def\BB{\mathbb B}
\def\OD{\textnormal{OD}_d}
\def\RR{\mathbb R}
\def\ZZ{\mathbb Z}
\def\SS{\mathbb S}
\def\NN{\mathbb N}
\def\DD{\mathbb D}
\def\H{\mathcal{H}}

\def\HTB{\textnormal{HT}\BB}
\def\HTD{\textnormal{HTD}_d}
\def\DTD{\textnormal{DT}\D}
\def\TCD{\textnormal{TC}\D}
\def\CD{\textnormal{C}\D}
\def\TH{\textnormal{TH}}
\def\Hc{\textnormal{H}}
\def\Sym{\textnormal{Sym}}
\def\id{\textnormal{id}}
\def\int{\textnormal{int}}
\def\Lk{\textnormal{lk}}
\def\Star{\textnormal{star}}
\def\Image{\textnormal{Im}}

\newcommand{\rom}[1]{\uppercase\expandafter{\romannumeral #1\relax}}

\begin{abstract}
   We introduce asymptotically rigid mapping class groups of handlebodies and determine their finiteness properties, which vary depending on the space of ends of the underlying handlebody. As it turns out, in some cases, the homology of these groups coincides with the stable homology of handlebody groups, as studied by Hatcher and Wahl.
\end{abstract}

\maketitle

\textit{Asymptotically rigid mapping class groups} of manifolds have garnered significant attention in recent years (see \cite{Funar,Fun-Kap1, Fun-Kap2, BrThom}) for a variety of reasons. First, they sometimes appear as extensions of \textit{Higman-Thompson groups} by direct limits of mapping class groups of compact manifolds. In certain cases, this implies that their homology coincides with the \textit{stable homology} of the underlying manifold (see \cite[Section 10]{Asymp}). Additionally, in the context of surfaces, these groups form countable subgroups of the so-called\textit{ big mapping class groups}, an active area of research at the intersection of geometric group theory, low-dimensional topology, topological groups, and related fields.

A major motivation for studying asymptotically rigid mapping class groups stems from the study of their finiteness properties (see \cite{Asymp, SurfaceHoughton, BrThom, Fun-Kap1, Fun-Kap2, Ara-Fun}). Recall that a group is said to be of \textit{type $F_n$} if it admits a classifying space with finite $n$-skeleton, and is of \textit{type $F_\infty$} if it is $F_n$ for every $n$. A group $G$ is said to be of \textit{type $FP_n$} if $\ZZ$, viewed as a $G$-module, has a projective resolution whose first $n$ terms are finitely generated.

Notably, the finiteness properties of asymptotically rigid mapping class groups are often determined by the space of ends of the underlying manifold. Specifically, as shown in \cite{Asymp}, under certain general assumptions on the manifolds involved, if the end space of the underlying manifold is a Cantor set, then the associated asymptotically rigid mapping class group is of type $F_\infty$. In stark contrast, when the end space consists of $r$ isolated points, the group is of type $F_{r-1}$  but not $FP_r$, as shown in \cite{SurfaceHoughton}.

The objective of this paper is to extend the results of  \cite{Asymp, SurfaceHoughton} to the setting of handlebodies. To state our results, we provide a rough description of the relevant objects, postponing technical details to Section \ref{Sect1}. Let $O$ be any handlebody and $Y$ a genus-one handlebody, referred to as a \textit{piece}, with $d+1 \geq 2$ spots. The \textit{tree handlebody} $\T_{d,r}(O,Y)$ is constructed by first performing boundary-connect sums of $r$ copies of $Y$ to $O$ along the spots of $O$, and then inductively boundary-connect summing $d$ copies of $Y$ to the manifold obtained in the previous step. Notably, the end space of $\T_{d,r}(O,Y)$ is a Cantor set if $d\geq2$, and consists of $r$ isolated points if $d=1$. The \textit{asymptotically rigid handlebody group} $\H_{d,r}(O,Y)$ is the group of homeomorphisms of $\T_{d,r}(O,Y)$, up to isotopy, that \textit{eventually} map pieces to pieces in a \textit{trivial} manner, see Section \ref{Sect1} for precise definitions.

Our first theorem establishes an analog of \cite{Asymp} in this setting, specifically for the case where the end space is a Cantor set:

\begin{theorem}\label{Main1}
Given any handlebody $O$, a genus-one handlebody $Y$, $d\geq 2$ and $r\geq1$, the group $\H_{d,r}(O,Y)$ is of type $F_\infty$.
\end{theorem}

Our second result addresses the case where the end space consists of $r$ isolated points, serving as a handlebody analog of the main result of \cite{SurfaceHoughton}:

\begin{theorem}\label{Main2}
Given any handlebody $O$, a genus-one handlebody $Y$, and $r\geq1$, the group $\H_{1,r}(O,Y)$ is of type $F_{r-1}$ but not $FP_{r}$.
\end{theorem}

Following the approach of \cite{Fun-Kap2}, we deduce that in certain cases, the homology of $\H_{d,r}(O,Y)$ coincides with the stable homology of the handlebody groups, as studied by Hatcher and Wahl \cite{Disk}.

\begin{corollary}\label{Corst2}

Let $O$ be any handlebody, and let $Y$ be a genus-one handlebody. Let $M_g$ be any handlebody of genus $g\geq 2i+4$. Then 
$$H_i(\H_{2,1}(O,Y))\cong H_i(\H(M_g)).$$

\end{corollary}

\textbf{Plan of the paper.} 
Section \ref{Sect1} contains the definitions of $\T_{d,r}(O,Y)$ and 
$\H_{d,r}(O,Y)$, along with the base results on handlebodies and handlebody groups we will need. Section \ref{Sect2} introduces Brown's Theorems, which are used to prove the finiteness properties in Theorems \ref{Main1} and \ref{Main2}, by means of a cube complex on which $\H_{d,r}(O,Y)$ acts. The premises of these theorems are verified throughout Sections \ref{Sect2},  \ref{Sect3}, \ref{Sect3.5} and \ref{Sect4}. Particularly, Section \ref{Sect3} studies the connectivity of the \textit{descending links} in the cube complex using a series of simplicial complexes introduced there, Section \ref{Sect3.5} finishes the proof of Theorem \ref{Main1}, and Section \ref{Sect4} does so for Theorem \ref{Main2}. Section \ref{Sect5} studies the homology of asymptotically rigid handlebody groups, and proves Corollary \ref{Corst2}.

\textbf{Acknowledgments.} The author is indebted to Javier Aramayona for proposing the research questions, providing valuable insights into asymptotically rigid mapping class groups, and offering guidance on various aspects of this work. Thanks are due to Federico Cantero for his constructive suggestions, which contributed significantly to this paper. Appreciation is extended to Natalie Wahl for her insightful input on key points in Lemma \ref{m-lema} and Theorem \ref{Teo-dtd}. Gratitude is also expressed to Pablo Sánchez-Peralta for helpful conversations. The author acknowledges financial support from
the grant CEX2019-000904-S funded by MCIN/AEI/10.13039/501100011033
and from the grant PGC2018-101179-B-I00.

\section{Tree handlebodies and asymptotically rigid handlebody groups}\label{Sect1}

This section is dedicated to introducing \textit{tree handlebodies} and their \textit{asymptotically rigid handlebody groups}, following ideas from \cite[Section 3]{Asymp} and \cite[Section 2]{SurfaceHoughton}.

Let $N_1$ and $N_2$ be 3-manifolds with connected boundary. Choose any pair of closed discs $D_i\subset\partial N_i$ for $i=1,2$, an orientation-preserving homeomorphism $\mu_1:\DD^2\rightarrow D_1$, and an orientation-reversing homeomorphism $\mu_2:\DD^2\rightarrow D_2$. The \textit{boundary connected sum} $N_1\natural N_2$ is the result of identifying $D_1$ with $D_2$ via the homeomorphism $\mu_1\circ\mu_2^{-1}$. Up to homeomorphism, the result of such identification does not depend on the election of the $D_i$'s or the $\mu_i$'s.

\subsection{Handlebodies and handlebody groups}
This subsection gathers the theory of handlebodies that will be used throughout the paper. For further details, we refer the reader to \cite{Hensel}.

A \textit{solid torus} $T$ is 
the product $T:=\SS^1\times\DD^2$ of a circle and a closed disc. A \textit{handlebody} is the boundary connected sum $M=T_1\natural...\natural T_g\natural \DD^3$ with $g\in \NN$, and each $T_i$ a solid torus. The number $g$, called the \textit{genus} of $M$, is a homeomorphism invariant of the handlebody. We always require the genus of a handlebody to be greater than zero.

A \textit{spotted handlebody} $(M,A)$ is a handlebody $M$ together with a (possibly empty) subset $A\subset\partial M$ of \textit{spots}, i.e. pairwise disjoint closed discs. Given a spotted handlebody $(M,A)$, a \textit{spotted subhandlebody} $(N,B)$, is a spotted handlebody $N\subset M$ such that:

\begin{itemize}
    \item $\overline{\partial N\cap \int(M)}$ is a collection of closed discs, 
    \item for every spot $D \subset A$ either $D\cap \partial N=\emptyset$ or $D\subset \partial N$,
    \item $B=(\overline{\partial N\cap \int(M)})\cup (\partial N \cap A)$, i.e. the subhandlebody inherits spots from both the new discs formed in $\overline{\partial N\cap \int(M)}$ and discs in $A$ which are in $\partial N$.
\end{itemize}

Observe that the boundary connected sum of two spotted handlebodies is a spotted handlebody. We require homeomorphisms between spotted handlebodies to send spots to spots. 

The \textit{handlebody group} $\H(M,A)$ of the spotted handlebody $(M,A)$ is the group of homeomorphisms of $M$ fixing $A$ pointwise, up to isotopies fixing $A$ pointwise. We will often omit the set $A$ from notation for the sake of simplicity. 

Given a spotted handlebody $M$ with set of spots $A$, fix a set of parametrizations of its spots, i.e. a set of orientation-preserving homeomorphisms $\{\mu_i:\DD^2\rightarrow D_i\}_{D_i\subset A}$. The \textit{spot-permuting handlebody group} $\H_o(M,A)$ of the spotted handlebody $(M,A)$ is the group of isotopy classes (relative to $A$) of homeomorphisms of $M$ that respect these parametrizations, meaning that for any $\phi\in\H_o(M,A)$, there is a permutation $\sigma$ such that $\phi\circ \mu_i=\mu_{\sigma(i)}$ $\forall i$. Note that $\H(M,A)$ is a finite index subgroup of $\H_o(M,A)$.

Handlebody groups can be studied as subgroups of \textit{mapping class groups}. Given a compact surface $S$ we define the mapping class group $\Map(S)$ as the group of homeomorphisms of the surface fixing $\partial S$ pointwise, up to isotopies fixing $\partial S$ pointwise. Write $S_{(M,A)}:=\overline{\partial M\setminus A}$, which is a connected surface with boundary. There is a well-defined, injective map $$\R:\H(M,A)\rightarrow \Map(S_{(M,A)})$$ given by restriction to the boundary (see \cite[Section 3]{Hensel}). 

Given $p\in \partial M\setminus A$, $\H(M,A,p)$ is the group of homeomorphisms of $M$ fixing $A$ and the point $p$ pointwise, up to isotopies fixing $A$ and $p$ pointwise. The following exact sequences are the restriction to handlebody groups of the \textit{capping homomorphism} and the \textit{Birman exact sequence} (see \cite[Proposition 3.19, Theorem 4.6]{MCG}, \cite[Section 3]{Hensel}).

\begin{lem}\label{lem-sec} Given a spotted handlebody $(M,A)$, choose a spot $D\subset A$ and a point $p\in \int{(D)}$, write $A'=A\setminus {D}$. There are exact sequences

$$1\rightarrow \ZZ  \xrightarrow{\twist}  \H(M,A) 
 \xrightarrow{\capp} \H(M,A',p)\rightarrow 1,$$
$$1\rightarrow  \pi_1(S_{(M,A')},p) \xrightarrow{\push} \H(M,A',p) \xrightarrow{\forget} \H(M,A') \rightarrow1.$$ 

\end{lem}

\subsection{Tree handlebodies}\label{Subsect1.2}
Following \cite{Asymp} and \cite{SurfaceHoughton}, in this subsection we introduce \textit{tree handlebodies} which, loosely speaking, are the result of gluing handlebodies via boundary connected sum in a tree-like manner.

For $d\geq 1$ let $Y^d$ be a genus-one spotted handlebody with $d+1$ spots. We enumerate the spots in $Y^d$ as $A_0,A_1,...,A_d$ and fix an orientation-preserving homeomorphism $\mu_0:\DD^2\rightarrow A_0$, and orientation-reversing homeomorphisms $\mu_i:\DD^2\rightarrow A_i$ for $i\geq 1$.
Let $O$ be a spotted handlebody with spots $B:=\{B_1,...B_r\}$. Fix orientation-reversing homeomorphisms $\nu_i:\DD^2\rightarrow B_i$ for each spot. We set $O_0=O$, and define a sequence $\{O_k\}_{k\geq 1}$ of spotted handlebodies in the following way:

\begin{enumerate}
    \item $O_1$ is the spotted handlebody that results from gluing $r$ copies of $Y^d$ to $O_0$, via boundary connected sum of each $Y^d$ with $O_0$ using the maps $\mu_0\circ\nu_i^{-1}$ for the gluing.
    
    \item For $k\geq 2$ each $O_k$ is the spotted handlebody that results from gluing a copy of $Y^d$ to each spot in $O_{k-1}$, via boundary connected sum of each $Y^d$ with $O_{k-1}$ using the maps $\mu_0\circ\mu_i^{-1}$ for the gluing.
\end{enumerate}

\begin{figure}[ht]
\includegraphics[width=11cm]{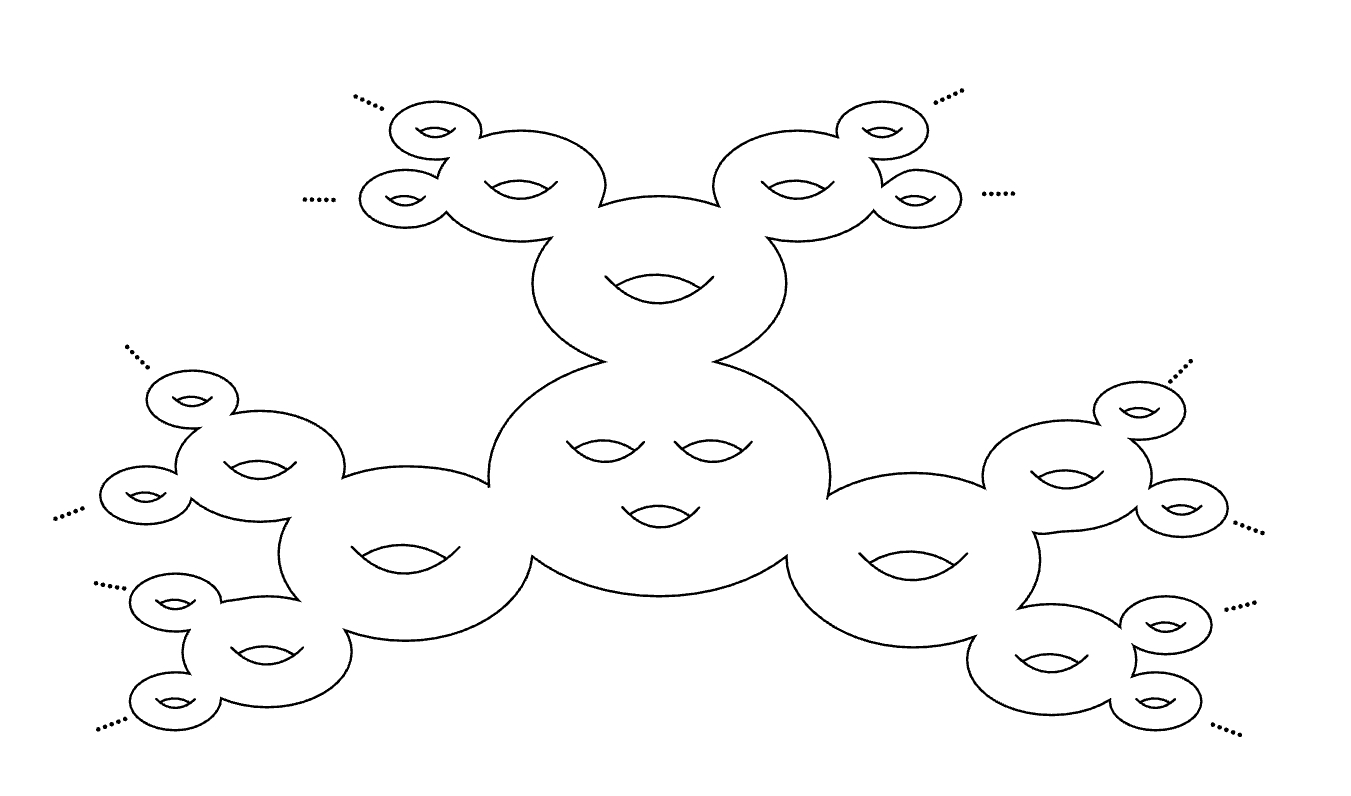}
\centering
\caption{The handlebody $O_4$, for $O$ a genus 3 handlebody, $d=2$, and $r=3$.}
\label{Ok}
\end{figure}

With respect to the above terminology we define the \textit{tree handlebody} $\T_{d,r}(O,Y)=\cup_{k=1}^\infty O_k$. Each of the connected components of $\overline{O_k\setminus O_{k-1}}$ for $k\geq 1$ is called a \textit{piece}. By construction, each piece is homeomorphic to $Y^d$. A spotted subhandlebody $M\subset\T_{d,r}(O,Y)$ is a \textit{suited handlebody} if it is connected and is the boundary connected sum of $O_0$ and finitely many pieces.

\begin{figure}[H]
\includegraphics[width=8cm]{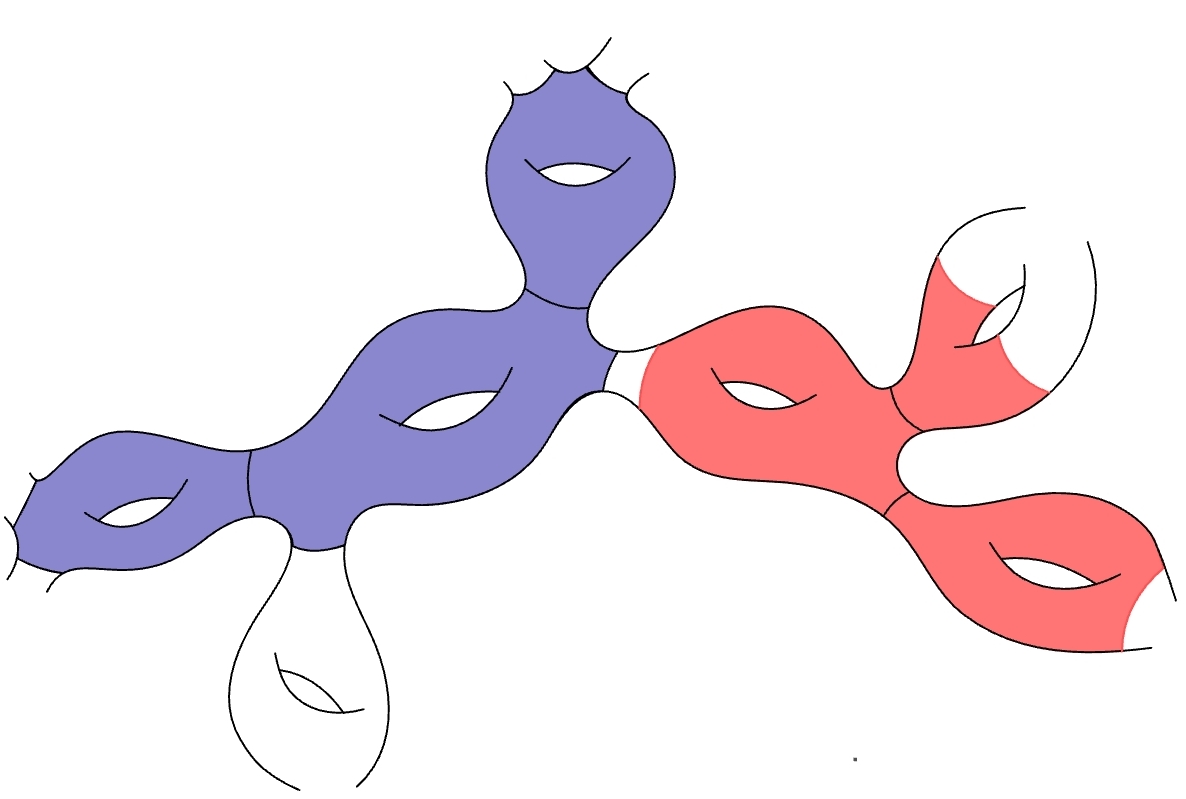}
\centering
\caption{In blue, a suited handlebody of $\T_{2,4}(O,Y)$, with $O$ a solid torus. In red, a spotted subhandlebody which is not suited.}
\label{subhandle}
\end{figure}

\subsection{Two types of tree handlebodies}\label{subsect1.3}

Tree handlebodies $\T_{d,r}(O,Y)$ with $d\geq2$ will be referred as \textit{Cantor handlebodies}, as their \textit{space of ends} of $\T_{d,r}(O,Y)$ is a Cantor set. We would like to remark that the boundary of a Cantor handlebody (as a manifold with boundary) is a Cantor surface in the terminology of \cite[Definition 3.1]{Asymp}. 

On the other hand, tree handlebodies $\T_{d,r}(O,Y)$ with $d=1$ will be referred as \textit{star handlebodies}. The space of ends of the star handlebody is a finite set of $r$ points.

\begin{figure}[ht]
\includegraphics[width=10cm]{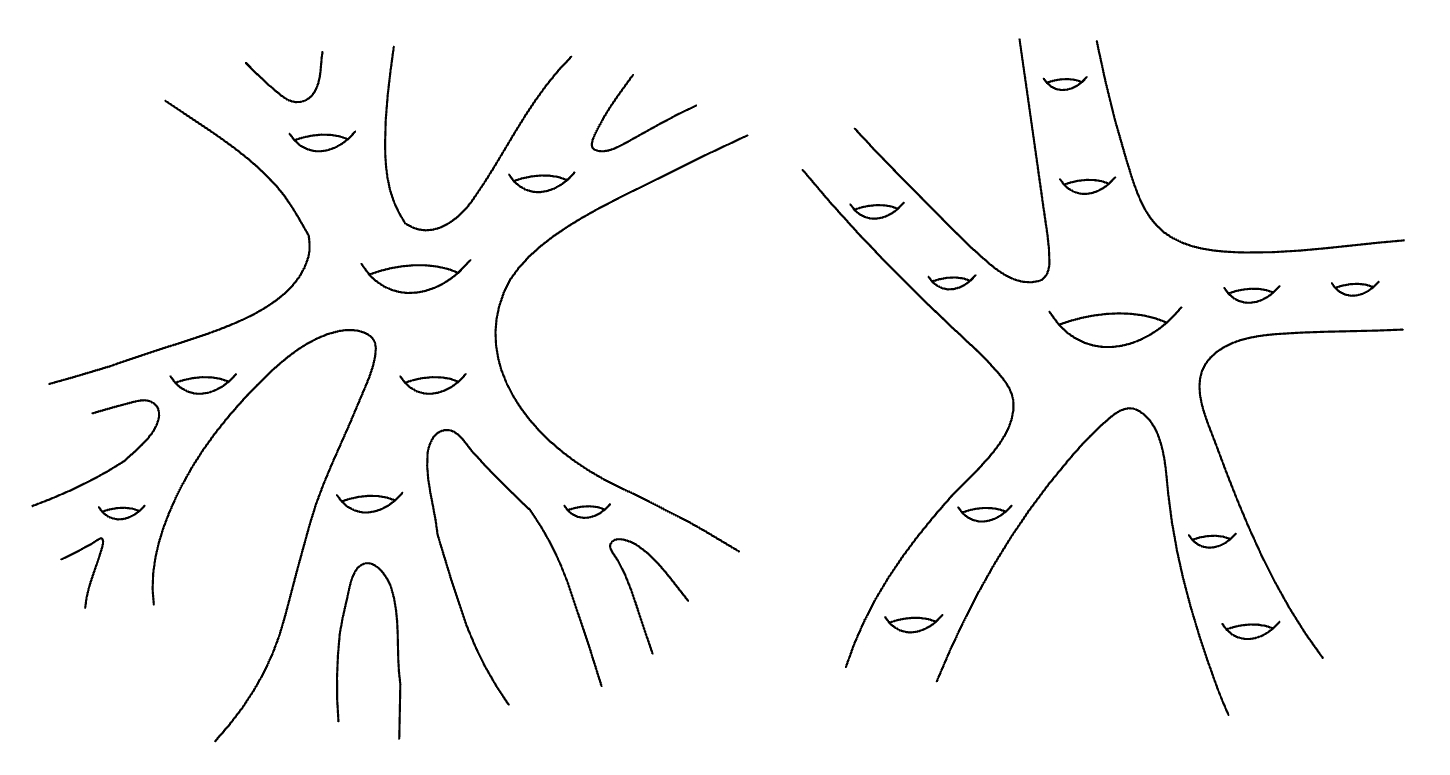}
\centering
\caption{A Cantor handlebody (left) and a star handlebody (right).}
\label{treevs}
\end{figure}

\subsection{Asymptotically rigid handlebody groups}\label{subsect1.4}

We now adapt the definition of rigid structure of \cite[Definition 3.8]{Asymp} to the setting of handlebodies. For each piece $Z$ of $\T_{d,r}(O,Y)$, choose a homeomorphism $i_Z:Z\rightarrow Y^d$ that agrees with the $\mu_i$'s used for the boundary connected sums.  

\begin{defi} (Preferred rigid structure). The set $\{i_Z\: : \: Z \text{ is a piece}\}$ is called the \textit{preferred rigid structure} on $\T_{d,r}(O,Y)$.
\end{defi}

\begin{defi}\label{rigid}(Asymptotically rigid homeomorphism).
Let $\T_{d,r}(O,Y)$ be a tree handlebody equipped with its preferred rigid structure. A homeomorphism $f:\T_{d,r}(O,Y)\rightarrow \T_{d,r}(O,Y)$ is \textit{asymptotically rigid} if there exists a suited handlebody $M\subset \T_{d,r}(O,Y)$ such that

\begin{enumerate}
    \item $f(M)$ is suited,
    
    \item $f$ is \textit{rigid outside of $M$}, that is, for every piece $Z$ outside of $M$, $f(Z)$ is also a piece and $i_{f(Z)}\circ f_{|Z}\circ i^{-1}_Z=\id_{Y^d}$.
\end{enumerate}

\end{defi}

The suited handlebody $M$ is called a \textit{defining handlebody} of $f$. Observe that any suited handlebody containing $M$ will be as well a defining handlebody for $f$. The composition $f_2\circ f_1$ of two asymptotically rigid self-homeomorphisms $f_1$ and $f_2$ is always defined up to enlarging the defining handlebody of $f_2$ so that it contains the image of the defining handlebody of $f_1$. This justifies the following definition:

\begin{defi}\label{def-Asy}(Asymptotically rigid handlebody group)
The \textit{asymptotically rigid handlebody group} $\H_{d,r}(O,Y)$ is the set of isotopy classes of asymptotically rigid self-homeomorphisms of $\T_{d,r}(O,Y)$, with the composition as its operation.
\end{defi}

To simplify notation, we will often blur the difference between elements of the asymptotically rigid handlebody groups and their representatives.

\begin{rem} An asymptotically rigid homeomorphism is determined by its action on any defining handlebody, hence one may be tempted to consider asymptotically rigid homeomorphisms up to \textit{proper isotopies}, i.e. isotopies supported in suited handlebodies. However, if two asymptotically rigid self-homeomorphisms of $\T_{d,r}$ are isotopic then they are also properly isotopic by a combination of \cite[Proposition 3.11]{Asymp} and \cite[Lemma 3.1]{Hensel}, so isotopy and proper isotopy can be used interchangeably. We use isotopy over proper isotopy in Definition \ref{def-Asy} to stress that $\H_{d,r}(O,Y)$ is a subgroup of the mapping class group $\Map(\T_{d,r}(O,Y))$.
\end{rem}

\subsection{Relation with mapping class groups} \label{relation}

Asymptotically rigid handlebody groups are the analogs, for handlebodies, of the \textit{asymptotic mapping class groups} $\B_{d,r}(O,Y)$ in \cite{Asymp} (when $d\geq2$), and of the \textit{surface Houghton groups} $\B_n$ in \cite{SurfaceHoughton} (when $d=1$). The restriction of elements of $\H_{d,r}(O,Y)$ to the boundary surface of the tree handlebody gives elements of $\B_{d,r}(O,Y)$ when $d\geq 2$; and elements of $\B_r$ when $d=1$ and $O$ has $g=0$, hence there are maps $$\R:\H_{d,r}(O,Y)\rightarrow\B_{d,r}(O,Y) , \; d\geq 2, $$ $$\R:\H_{d,r}(O,Y)\rightarrow\B_r, \; d=1,$$ defined by restriction to the boundary, which are well defined, as two representatives of the same element of  $\H_{d,r}(O,Y)$ define isotopic elements of the mapping class group of the surface $\partial\T_{d,r}(O,Y)$. The map $\R$ is a monomorphism because of \cite[Lemma 3.1, Section 3]{Hensel}.

However, $\R$ cannot be used to deduce the finiteness properties of asymptotically rigid handlebody groups, as the group $\H_{d,r}(O,Y)$ has infinite index as a subgroup of $\B_{d,r}(O,Y)$ or $\B_n$, which can be deduced from \cite[Corollary 5.4]{Hensel}.

\subsection{Compactly supported handlebody group}\label{CompacSup}

The \textit{compactly supported handlebody group} $\H_{d,r}^c(O,Y)$ is the subgroup of $\H_{d,r}(O,Y)$ of elements with compact support, meaning that they are the identity outside of a compact subhandlebody. Recall the handlebodies $O_k$ in the definition of $\T_{d,r}(O,Y)$. Using the fact that every inclusion $O_k\subset O_{k+1}$ induces a homomorphism $j_{O_k,O_{k+1}}:\H(O_k)\rightarrow\H(O_{k+1})$, we can equivalently define the group as a direct limit: $$\H^c_{d,r}(O,Y):=\underset{\longrightarrow} \lim \, \H(O_k).$$

\subsection{Relation with Higman-Thompson groups}\label{SubRelThom}

Analogously to \cite[Subsection 4.1]{Asymp}, there is a sequence relating asymptotically rigid handlebody groups of Cantor handlebodies with \textit{Higman-Thompson groups} (\cite{Higman}, \cite[Section 4]{Brown}), denoted by $V_{d,r}$ with $d\geq 2$ $r \geq 1$. Let $\T_{d,r}(O,Y)$ be any tree handlebody with $d\geq 2$. The action of an element of $\H_{d,r}(O,Y)$ on the space of ends of the tree induces a homeomorphism of the Cantor set, which as in \cite[Proposition 4.3]{Asymp} gives rise to a short exact sequence:
$$1\rightarrow\H_{d,r}^c(O,Y)\xrightarrow{}\H_{d,r}(O,Y)\xrightarrow{} V_{d,r}\rightarrow 1.$$

\section{Cube complex}\label{Sect2}

The key ingredient to proving Theorem \ref{Main1} and Theorem \ref{Main2} is a classical criterion due to Brown \cite{Brown}. The language here used is the one of \textit{discrete Morse theory} as written in \cite{MorseTh}. We give a brief summary, and refer the reader to \cite{MorseTh},  \cite[Appendix 1]{Asymp} for further details.

\subsection{Brown's criterion}

Let $G$ be a group acting cellularly and cellwise isometrically on a piecewise euclidean $CW$-complex $\K$ (i.e. a $CW$-complex where each closed cell is isometric to a convex polyhedron embedded in Euclidean space), equipped with a $G$-invariant \textit{Morse function}, that is, a $G$-invariant map $h:V(\K)\rightarrow \RR$  defined on the vertices $V(\K)$ of $\K$, with a unique maximum per cell, which we denote the \textit{top vertex}. The Morse function can be extended linearly to a map $h:\K\rightarrow \RR$ defined on the whole of $\K$. Given $v\in V(\K)$, a \textit{descending cell} of $v$ is a cell with $v$ as top vertex. We refer to $h(v)$ as the \textit{height} of $v$. Images of vertices of $\K$ under $h$ are denoted \textit{critical values}. The map $h$ is said to be \textit{discrete} if the set of critical values is discrete. Denote by $\K^{\leq s}$ the subcomplex of $\K$ spanned by vertices of height at most $s$.

Let $c$ be a coface of $v$, i.e. a cell with $v\in c$. The set of \textit{directions} based at $v$ and pointing into $c$ is the quotient of the set of straight line segments with $v$ as an endpoint and its other endpoint in $c$ under the following equivalence relation: two segments are equivalent if one is an initial segment of the other. The set of directions from $v$ into $c$ forms an spherical polyhedron $\Lk(v,c)$. The \textit{link} $\Lk(v,\K)$ is the result of identifying the links $\Lk(v,c)$ of every coface $c\in\K$ of $v$ along their intersections. The \textit{descending link} $\dl(v,\K)$ of a vertex $v$ is the link of $v$ in the union of all its descending cells.

We say a group $G$ is of type $F_n$ if it acts properly discontinuously, freely and cocompactly on a CW-complex with vanishing homotopy groups $\pi_0,...,\pi_{n-1}$. A group is of type $F_\infty$ if it is of type $F_n$ for all $ n> 0$. The group $G$ is said to be of type $FP_n$ if the $G$-module $\ZZ$ admits a projective resolution whose first $n$ terms are finitely generated. Being of type $F_n$ implies being of type $FP_n$ (see \cite[Section 8.2]{Geoghegan}).

The version of the criterion we introduce can be deduced from \cite[Subsection 3.1]{SurfaceHoughton} and its proof. We split it into two results as in this way, Theorem \ref{brown1} will apply to Cantor handlebodies and Theorem \ref{brown2} to star handlebodies.

\begin{theor}\label{brown1} (Brown's criterion I) Let $G$ be a group acting cellwise by isometries on a piecewise euclidean $CW$-complex $\K$. Assume $\K$ is equipped with a discrete $G$-invariant Morse function $h:\K\rightarrow \RR$. Suppose that:

\begin{enumerate}
    \item $\K$ is contractible.
    \item The quotient of $\K^{\leq s}$ by $G$ is finite for all critical values s.
    \item Every cell stabilizer is of type $F_\infty$.
    \item There exists $l\geq 1$ such that, for a sufficiently large critical value $s$ and for every vertex $v\in \K$ with $h(v)\geq s$, $\dl(v,\K)$ is $(l-1)$-connected.
\end{enumerate}

Then $G$ is of type $F_l$.
\end{theor}

Particularly, if condition $(4)$ is satisfied for any $l$, the group is of type $F_\infty$.

\begin{theor}\label{brown2}
    (Brown's criterion II) Let $G$ be a group and $\K$ be a $CW$-complex satisfying every hypothesis in Theorem \ref{brown1}. If, additionally, the following two hypotheses are satisfied:

    \begin{enumerate}
    \setcounter{enumi}{4}

        \item There exists a critical value $s$ such that for every vertex $v\in \K$ with $h(v)\geq s$, $\dl(v,\K)$ is of dimension $l$.

        \item For each critical value $s$ there exists a vertex $v$ with $h(v)\geq s$ such that $\dl(v,\K)$ is non-contractible.
    \end{enumerate}

Then $G$ is op type $F_l$ but not of type $FP_{l+1}$.
    
\end{theor}

In the next subsection, we will introduce a cube complex $\Xc_{d,r}(O,Y)$, where $\H_{d,r}(O,Y)$ acts cellwise by isometries. The proofs of Theorems \ref{Main1} and \ref{Main2} then boil down to checking a series of premises depending on whether we are in the Cantor handlebody case (with $d\geq 2$), or in the star handlebody case (with $d=1$):

\textbf{Cantor case (Theorem \ref{Main1}):} Condition (1) and (2) in Theorem \ref{brown1} are proved in Corollary \ref{Cond1},  and condition (3) in Corollary \ref{cond3}. Section \ref{Sect3} studies connectivity bounds for the descending links $\dl(v,\K)$. Particularly, Corollary \ref{TheorI} states that condition (4) in Theorem \ref{brown1} is satisfied for every $l\geq0$, hence checking that the action of $\H_{d,r}(O,Y)$ on $\Xc_{d,r}(O,Y)$ satisfies every hypothesis of Theorem \ref{brown1} for every $l\geq0$. This will prove Theorem \ref{Main1}.

\textbf{Star case (Theorem \ref{Main2}):} Condition (1), (2) and (3) in Theorem \ref{brown1} are proved as in the Cantor case. Corollary \ref{Cond(4)} proves Condition (4) in Theorem \ref{brown1} for $l=r-1$. Condition (5) and (6) in Theorem \ref{brown2} are proven in Lemma \ref{lem(5)} and Lemma \ref{lem(6)} respectively. This will prove Theorem \ref{Main2}.

\subsection{Stein-Farley cube complex}
The goal of this section is to introduce the complex $\Xc_{d,r}(O,Y)$, which is analogous to the cube complex introduced in  \cite[Section 5]{Asymp}, and also closely resembles that introduced by Genevois-Lonjou-Urech in \cite{BrThom}. The results in this section can be deduced from \cite[Sections 5, 6] {Asymp} bearing in mind the difference in context, as here we deal with boundary connected sums instead of connected sums. We omit $d, r, O$, and $Y$ from the notation for simplicity, and write $\T=\T_{d,r}(O,Y)$, $\Xc=\Xc_{d,r}(O,Y)$ and $\H=\H_{d,r}(O,Y)$.

\begin{rem} As in \cite{Asymp}, a key element in the construction of the complex $\Xc$ is the fact that handlebodies satisfy the inclusion, intersection, and cancellation properties, as shown in Appendix \ref{Ap1}. The places where these properties are used will be highlighted in the text.
\end{rem}

Consider all ordered pairs $(M,f)$ where $M$ is a suited handlebody of $\T$ and $f\in\H$. We deem two such pairs $(M_1,f_1)$ and $(M_2,f_2)$ to be equivalent, and write $(M_1,f_1)\sim (M_2,f_2)$, if and only if there are representing homeomorphisms (abusing notation) $f_1$ and $f_2$ such that $f_2^{-1}\circ f_1$ maps $M_1$ onto $M_2$ and is rigid outside of $M_1$. We denote by $[(M,f)]$ the equivalence class of the pair $(M,f)$ with respect to this relation, and write $\P$ for the set of equivalence classes. Observe that $\H$ acts on $\P$ by left multiplication, namely $g\cdot[(N,f)]=[(N,g\circ f)]$.

Consider a pair $(M,f)$. Since $M$ is a suited handlebody, it is the union of finitely many pieces and $O_0$. We define the\textit{ height} $h((M,f))$ of the pair $(M,f)$ as the number of pieces in $M$. Note that if $(M_1,f_1)\sim (M_2,f_2)$ then $h((M_1,f_1))=h((M_2,f_2))$, and thus $h$ descends to a well-defined Morse function (abusing notation) $h:\P\rightarrow\NN$, setting the height $h([(M,f)])$ to be the height of any representative.

We introduce a relation $\preceq$ on the elements of $\P$ by declaring $x_1\preceq x_2$ if and only if $x_1=[M_1,f]$ and $x_2=[M_2,f]$ for suited handlebodies $M_1\subset M_2$ such that $\overline{M_2\setminus M_1}$ is a disjoint union of pieces.

Define \textit{a closed interval} $[x,y]$ to be the set of elements $z\in\P$ such that $x\preceq z\preceq y$. The relation $\preceq$ can be used to construct a cube complex with $\P$ as its 0-skeleton and every $[x,y]$ as a $d$-cube with $d=h(y)-h(x)$ (See \cite[Proposition 5.11]{Asymp} for details). We will refer to the cube complex $\Xc$ as the \textit{Stein-Farley cube complex} associated to the tree handlebody $\T$. 

Observe that $\H$ acts on $\Xc$ respecting the cubical structure, so the action on $\Xc$ is cellular. In addition, the (discrete) function $h$ can be linearly extended over cubes to a Morse function over $\Xc$. A direct translation of the reasoning in \cite[Proposition 5.7]{Asymp} and \cite[Lemma 6.2]{Asymp} results in the following:

\begin{theor}
    The complex $\Xc$ is contractible, and the action of $\H$ on $\Xc^{\leq k}$ is cocompact for all $1\leq k < \infty$. 
\end{theor}

\begin{cor}\label{Cond1}
    The complex $\Xc$ and the group $\H$ satisfy conditions (1) and (2) in Theorem \ref{brown1}.
\end{cor}

\subsection{Cell stabilizers} 
To check that cell stabilizers of the action of $\H$ on $\Xc$ satisfy condition (3) in Theorem \ref{brown1}, we relate them to the \textit{spot-permuting handlebody group} of the suited handlebodies. Adapting the argument in \cite[Lemma 6.3]{Asymp} we get the following result:

\begin{lem}\label{lem-index}
    The cube stabilizers of the action of $\H$ on $\Xc$ are isomorphic to a finite index subgroup of the spot-permuting handlebody group  $\H_o(M,A)$ of suited handlebodies $(M,A)$.
\end{lem}

Lemma \ref{lem-index} is used in conjunction with the following result (see \cite[Section 7]{Geoghegan}).

\begin{lem}\label{Lem-sec2}

Let $G$, $K$ and $Q$ be groups.
\begin{enumerate}
    \item Let $1 \rightarrow K \rightarrow G \rightarrow Q \rightarrow 1$ be a short sequence. If $K$ and $Q$ (resp. $K$ and $G$) are of type $F_n$, then so is $G$ (resp. $Q$).
    \item If $K \leq G$ has finite-index, then $G$ is of type $F_n$ if and only if $K$ is.

\end{enumerate}
    
\end{lem}

Since handlebody groups $\H(M)$ are of type $F_\infty$ (by \cite[Theorem 1.1 a, Theorem 6.1]{McCullough}), part (1) of Lemma \ref{Lem-sec2}, applied to both sequences in Lemma \ref{lem-sec} implies that spotted handlebody groups $\H(M,A)$ are also $F_\infty$. Furthermore, since $\H(M,A)$ is a finite index subgroup of $\H_o(M,A)$, Lemma \ref{Lem-sec2} part (2) implies that $\H_o(M,A)$ is $F_\infty$, and that so are cell stabilizers. Hence we deduce the following corollary.

\begin{cor}\label{cond3}
    The action of $\H$ on $\Xc$ satisfies condition (3) in Theorem \ref{brown1}; that is, every cell stabilizer is of type $F_\infty$.
\end{cor}

\section{Connectivity of descending links} \label{Sect3}
This section aims to analyze the connectivity properties of descending links, and in particular to check condition (4) in Theorem \ref{brown1} and condition (5) in Theorem \ref{brown2}. As is the case in \cite{Asymp}, we will proceed by first reinterpreting descending links by a series of simplicial complexes built from topological objects on the underlying manifolds. We remark that, while the arguments used closely resemble those in \cite{Asymp}, the technicalities are rather different, coming from the fact that here we are dealing with boundary connected sums instead of connected sums. 

The simplicial complexes we will need are schematized in the following diagram, which we include for the convenience of the reader. They will be introduced, and their connectivity analyzed, in the following subsections. 

\begin{figure}[ht]
\begin{tikzcd}
 \D_c(M_g,A)& \arrow[l, ""]\DTD(M_g,A,I)\arrow[dl,hook, ""]&\TCD(M_g,A,I)\arrow[l,hook, "(D)" above]\arrow[r, "(B)(C)(D)"above] &\CD(M_g,A)\arrow[d, "(E)" left]\\
 \DTD^{+}(M_g,A,I)\arrow[u, "(A)(B)(C)"]&\HTD(M_g,A,B)\arrow[dd, ""]\arrow[r,hook, "(D)"]&\HTD^+(M_g,A,B)\arrow[dd,"(B)" left]&\Hc(M_g,A)\\
 &&\arrow[r,dashed, ""]&\TH(M_g,A,B)\arrow[u, "(B)(D)"]\\
 \dl(v,\Xc)\arrow[r, "(E)"]&\Pc(M_g,A,B)&\OD(M_g,A,B)&
\end{tikzcd}
\setcounter{figure}{0}
\refstepcounter{figure} % Updates the internal figure counter
\caption*{\sc Diagram \thefigure} % Uses the counter for reference
\label{diagr1}
\end{figure}

Throughout this section, we define the various complexes involved and prove their corresponding connectivity results. The labels (A), (B), (C), (D), (E), indicate the main results used to relate the connectivities of the different complexes involved, we will introduce them in the following subsection. 

\subsection{Connectivity tools} Given a topological space $X$, we say $X$ is \textit{(-$1$)-connected} if $X$ is nonempty, and \textit{$0$-connected} if $X$ is path-connected. We say $X$ is $n$-connected for $n\geq1$ if its $n$-th homotopy group is trivial. Following \cite[Section 8]{Quill}, a simplicial complex is \textit{weakly Cohen-Macaulay} (or wCM) of dimension $n$ if it is $(n-1)$-connected, and the link of every $d$-simplex is $(n-d-2)$-connected.

\textbf{(A) Barycentric subcomplexes:} Given a simplicial complex $X$, consider its barycentric subdivision $X'$. The complex $X_m$ is the subcomplex of $X'$ spanned by the vertices corresponding to simplices of dimension at least $m-1$.

\begin{theor}\label{teo-m}\cite[Lemma 3.8]{Disk} If $X$ is wCM of dimension $n$ then $X_m$ is $(n-m)$-connected.
\end{theor}

\textbf{(B) Quillen's Fiber Theorem:} The following theorem is stated in \cite[Proposition A.17]{Asymp}, and is a specification in terms of simplicial complexes of a more general result due to Quillen \cite[Proposition 7.6]{Quill}.

\begin{theor}\label{fiber} Let $X$ and $Y$ be simplicial complexes, $p:Y\rightarrow X$ a simplicial map and assume that the preimage of each closed simplex is $n$-connected. Then $X$ is $n$-connected if and only if $Y$ is $n$-connected.
    
\end{theor}

We will also need one more version of Quillen's Fiber Theorem, which can be deduced from \cite[Proposition 1.6]{Quill} and \cite[Lemma 2.6]{Tethers}.

\begin{theor}\label{fiber2}Let $p:Y\rightarrow X$ be a simplicial map. Define $p_m:Y_m\rightarrow X_m$ to be the maps induced by $p$. Suppose $p_m^{-1} (v)$ is contractible for each vertex $v\in X_m$. Then $p_m$ is a homotopy equivalence.
    
\end{theor}

\textbf{(C) Retraction into a subcomplex:}  The theorem we now introduce is a slight variation of \cite[Lemma 2.9]{Tethers}. Let $Y$ be a subcomplex of a simplicial complex $X$, with a function $c:V(X)\rightarrow \ZZ_{\geq0}$ taking positive values outside of $Y$, and value 0 on $Y$. Define the complexity of a simplex $\sigma\in X$ as $c(\sigma):=\underset{v\in\sigma}{\sum}c(v)$.

\begin{theor}\label{teoflow}
    Under the conditions above, suppose that to each vertex $v\in X\setminus Y$ we associate a vertex $\Delta v\in\Lk(v,X)$, and to each simplex $\sigma\in X$ with $\sigma\notin Y$ we associate one of the vertices $v_\sigma\in \sigma$ so that:
    \begin{enumerate}
        \item The join $\sigma*\Delta v_\sigma$ is a simplex of $X$.
        \item Let $\Delta\sigma=\Lk(v_\sigma,\sigma)*\Delta v_\sigma$ and $\Delta^k \sigma=\underbrace{\Delta \Delta...\Delta}_\text{k times}\sigma$ .Then if $ \sigma \notin Y$ there exists $k>0$ such that $c(\Delta^k \sigma)<c(\sigma)$.
        \item If $\tau$ is a face of $\sigma$ which contains $v_\sigma$ then $v_\tau=v_\sigma$.
    \end{enumerate}
Then $Y$ is a deformation retraction of $X$.
\end{theor}

\begin{proof}[\sc Proof of Theorem {\rm \ref{teoflow}}]
For each $\sigma=\langle v_\sigma, v_1,...,v_k \rangle \notin Y$, we define a flow on  $\sigma*\Delta v_\sigma$. In barycentric coordinates, given a point in $\sigma*\Delta v_\sigma$, we shift the weight from $v_\sigma$ to $\Delta v_\sigma$ while keeping the rest of the weights of the vertices invariant. If $\Delta v_\sigma\in \sigma$ then the flow sends any point in $\sigma$ to the face $\langle v_1,...,\Delta v_\sigma,...,v_k \rangle$. If $\Delta v_\sigma \notin \sigma$, points in $\sigma$ travel to $\langle \Delta v_\sigma, v_1,...,v_k \rangle$ by a flow of lines in $\sigma*\Delta v_\sigma$ (see Figure \ref{Sigmaflow}). 

\setcounter{figure}{3}
\begin{figure}[ht]
\includegraphics[width=10cm]{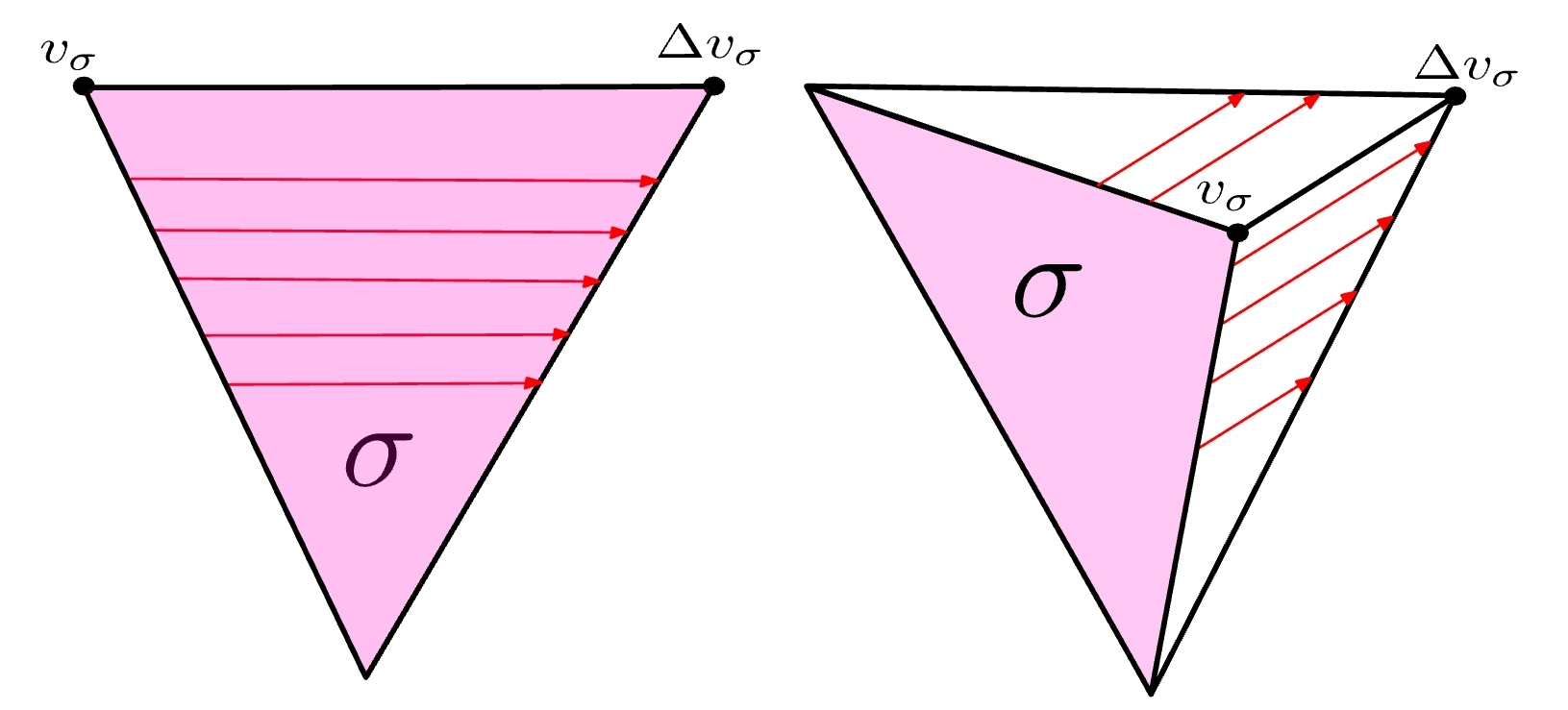}
\centering
\caption{A representation of the flow lines.}
\label{Sigmaflow}
\end{figure}

Condition (3) guarantees that the flow is continuous at every point, i.e. that the flows in different simplices agree in their common faces. Condition (2) guarantees that any point in $X$ reaches $Y$ after flowing through finitely many simplices, as the complexity of the simplices always decreases after a finite number of steps. So, if we make every point travel through the flow as to reach $Y$ in time 1, we get a continuous retraction from $X$ into $Y$.

\end{proof}

\textbf{(D) Bad simplex argument:} These concepts and results are introduced in \cite[Section 2.1]{Tethers}. Let $X$ be a simplicial complex. For each simplex $\sigma$ we select a subset of vertices $\Bar{\sigma}\subset\sigma$, designated as \textit{bad vertices}, such that:
\begin{enumerate}
    \item For $\sigma\subset\tau$, we have $\Bar{\sigma}\subset\Bar{\tau}$. (Monoticity)
    \item For any simplex, $\Bar{\Bar{\sigma}}=\Bar{\sigma}$. (Idempotence)
\end{enumerate}

A simplex $\sigma$ is \textit{good} if $\Bar{\sigma}$ is empty, and it is \textit{bad} if $\Bar{\sigma}=\sigma$. Any face of a good simplex is also good because of monoticity, so the set of good simplices form a subcomplex $X^{\text{good}}\subset X$ denoted the \textit{good complex}. The \textit{good link} $G_\sigma$ of a simplex $\sigma$ consists of proper cofaces $\tau$ of $\sigma$ such that $\Bar{\sigma}=\Bar{\tau}$. Because of monoticity, we deduce that $G_\sigma$ is a subcomplex of the link of $\sigma$. The following theorem is introduced in \cite[Section 2.1]{Tethers}. Here, we phrase it as in \cite[Proposition A.7]{Asymp}, where a proof of the theorem can be found.

\begin{theor}\label{badsim}(Bad simplex argument) Suppose there is $m\in\NN$ such that for all bad simplices $\sigma$, the good link $G_\sigma$ is $(m-\dim(\sigma))$-connected. Then, the inclusion $X^{\text{good}}\xhookrightarrow{} X$ induces an isomorphism in homotopy groups $\pi_d$ for $d\leq m$, and a  epimorphism for $d=m+1$.
\end{theor}

\textbf{(E) Join complex:} The concept of \textit{join complex} is introduced by Hatcher and Wahl in \cite[Definition 3.2]{Disk}.

\begin{defi}\label{def-join}
A \textit{join complex} over a simplicial complex $X$ is a simplicial complex together with a simplicial map $p:Y\rightarrow X$ satisfying the following properties:
    \begin{enumerate}
        \item $p$ is surjective,
        \item $p$ is injective on individual simplices,
        \item for each $p$-simplex $\sigma=[ x_0,...,x_p]$ of $X$ the subcomplex $Y(\sigma)$ of $Y$ consisting of all the $p$-simplices whose image is $\sigma$ is the join $Y_{x_0}(\sigma)*...*Y_{x_p}(\sigma)$ of the vertex sets $Y_{x_i}(\sigma)=Y(\sigma)\cap p^{-1}(x_i).$
    \end{enumerate}
    The sets $Y(\sigma)$ and $p^{-1}(\sigma)$ are not necessarily equal. When the inclusions $Y_{x_i}(\sigma)\subset p^{-1}(x_i)$ are all equalities we say that $Y$ is a \textit{complete} join over $X$.
\end{defi}

The following theorem is a combination of \cite[Remark A.14.]{Asymp} and \cite[Proposition 3.5]{Disk}.

\begin{theor}\label{join2}
    If $Y$ is a complete join complex over a simplicial complex $X$ of dimension $n$ then:
    
    \begin{itemize}
        \item if $X$ is wCM of dimension $n$ then so is $Y$,
        \item if $Y$ is $k$-connected with $k\leq n$, then $X$ is $k$-connected.
        \end{itemize}
\end{theor}

\subsection{The piece complex.}\label{subspiececpx} 
Each descending link $\dl{(v,\Xc)}$ can be studied as a complete join over a \textit{piece complex}, which we now introduce following the approach in \cite[Subsection 6.1]{Asymp}. Let $M_g$ be a spotted handlebody, $A$ its set of spots, and $B\subset A$ a subset of spots. 

By an \textit{embedded disc} in $M_g$ we mean an embedding $D:\DD^2\rightarrow M_g$ of a closed disc (or to its image, abusing terminology) with the condition that $\partial D\cap \int(M_g)=\emptyset$, $\int(D)\cap\partial M_g=\emptyset$, and $D\cap A=\emptyset$. 

The \textit{piece complex} $\Pc(M_g,A,B)$ is the simplicial complex whose vertices are isotopy classes of spotted subhandlebodies homeomorphic to $Y^d$ (not necessarily pieces), with one spot being an embedded disc $D$, and the other $d$ spots being elements of $B$; and whose $k$-simplices are collections of  $k+1$ vertices, with a set of pairwise disjoint representatives. When $A=B$ we omit $B$ from the notation and simply write $\Pc(M_g,A)$.

\begin{figure}[h]
\includegraphics[width=10cm]{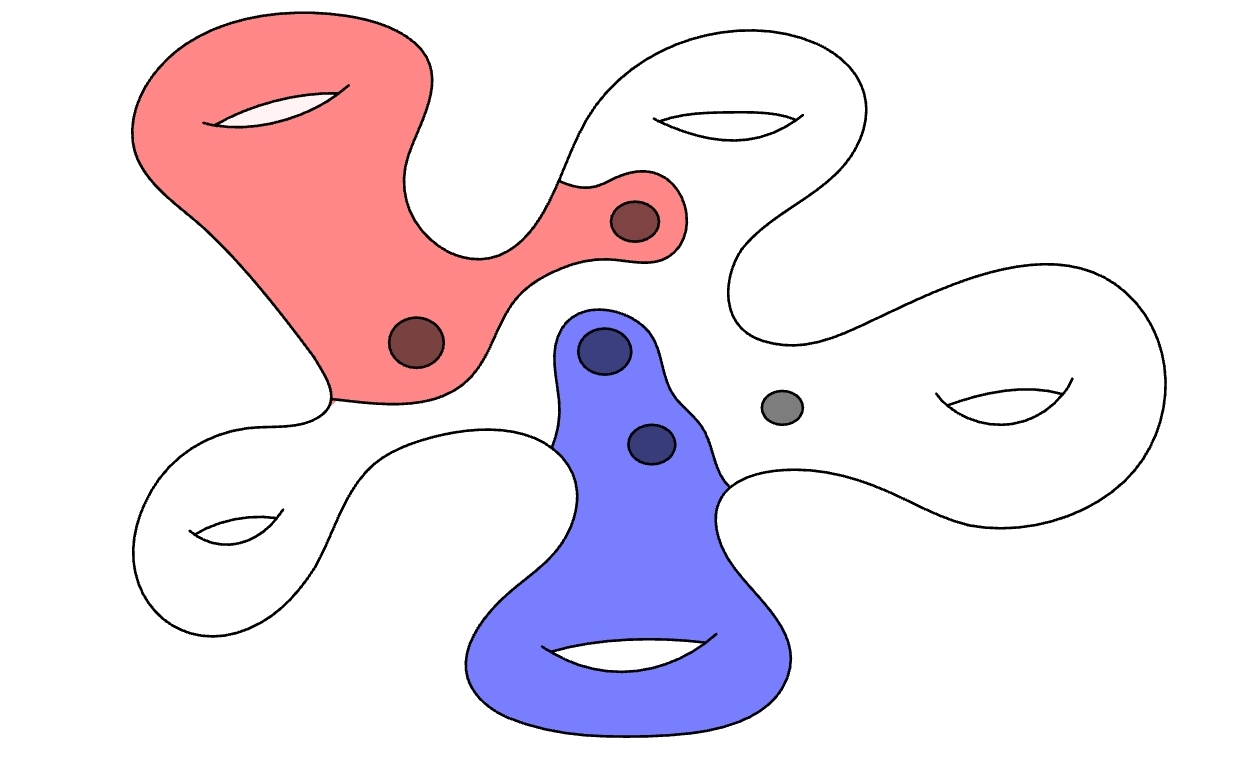}
\centering
\caption{A 1-simplex in $\Pc(M_5,A)$.}
\label{pieces}
\end{figure}

Let $v\in \Xc$. Up to the action of $\H$ we can write $v=[(M,\id)]$. A simplex $\sigma\in\dl(v,\Xc)$ can be represented as $(N,g)\in[(M,\id)]$ and $N'\subset N$ such that $\overline{N\setminus N'}$ is a collection $\{Y_0,...Y_k\}$ of disjoint pieces. There is a map:

$$\Pi: \dl(v,\Xc)\rightarrow \Pc(M_g,A)$$

where $A$ is the set of spots of $M$, and where
$\Pi(\sigma)=[ Z_1,...,Z_k]$, with each $Z_i= g(Y_i)$. 

The proof of the following proposition mimics \cite[Proposition 6.6]{Asymp}, relying on the key fact that handlebodies satisfy the cancellation property (see Appendix \ref{Ap1}). 

\begin{prop}\label{teo-linkpiece}
With respect to the map $\Pi$, the descending link $\dl(v,\Xc)$ is a complete join over the piece complex $\Pc(M_g,A)$.
\end{prop}

We will use Theorem \ref{join2} and Lemma \ref{teo-linkpiece} to prove a bound on the connectivity of $\dl(v,\Xc)$. The rest of the section is devoted to finding connectivity  bounds for $\Pc(M_g,A)$. Particularly we will prove:

\begin{theor}\label{teo-piece}
Let $m= \left\lfloor\min\left\{ \frac{g-3}{2}, \frac{|A|+1}{2d-1}-2,|A|-2\right\}\right\rfloor$. Then the piece complex $\Pc(M_g,A)$ is wCM of dimension $m+1$.
\end{theor}

Because of Proposition \ref{teo-linkpiece}, Theorem \ref{teo-piece}, and Theorem \ref{join2}, we deduce the following corollary: 

\begin{cor}\label{cor-lk}
The descending link $\dl(v,\Xc)$ of a vertex $v=[(X,\id)]$ is $m$-connected, with $$m=\min\left\{\left\lfloor\frac{g-3}{2}\right\rfloor,\left\lfloor \frac{|A|+1}{2d-1}-2\right\rfloor, |A|-2\right\},$$ where $|A|$ is the number of spots of $X$ and $g$ its genus.
\end{cor}

\begin{rem}\label{rem-3.11}
Let $v\in \Xc$. Although $r$ is not explicitly included in the notation for $\dl(v,\Xc)$, we emphasize that the connectivity of $\Pc(M_g,A)$, and thus that of $\dl(v,\Xc)$, does depend on $r$, since for a suited handlebody $(M,A)$ corresponding to a vertex $v$ of height $h(v)=k$, we have that $|A|=r+k(d-1)$.

\end{rem}

The structure of the rest of the section is as follows. Each subsection is dedicated to introducing the complexes in Diagram \ref{diagr1}, and to analyzing their connectivity.

\subsection{The coconnected disc system complex.} 
The \textit{coconnected disc system complex} $\D_c(M_g,A)$, introduced in \cite[Section 8]{Disk}, is the simplicial complex whose $k$-simplices are collections of $k+1$ isotopy classes of embedded discs with a set of pairwise disjoint representatives so that the complement of their union is connected (see Figure \ref{Coconnected}). We will blur the difference between isotopy classes and their representatives.

\begin{figure}[ht]
\includegraphics[width=10cm]{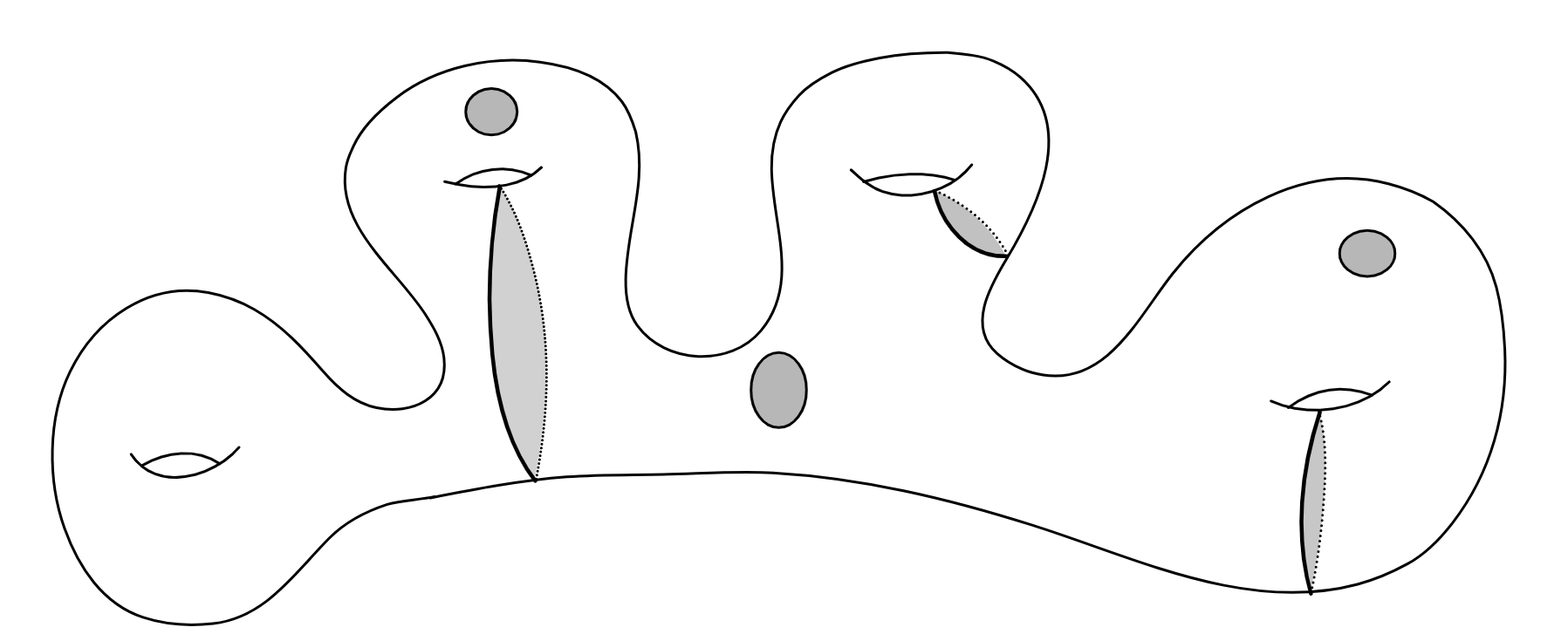}
\centering
\caption{A 2-simplex in $\D_c(M_4,A)$.}
\label{Coconnected}
\end{figure}

The next result can be found in \cite[Proposition 8.4]{Disk}:

\begin{lem}\label{DiscConect}
$\D_c(M_g,A)$ is wCM of dimension $g-1$.
\end{lem}

\subsection{The doubly-tethered disc complex.}

We start the subsection by giving some general notation necessary for defining the remaining complexes. Let $(M_g,A)$ be a spotted handlebody. An \textit{arc} $\alpha$ in $\partial M_g$ is a continuous map $\a:[0,1]\rightarrow \partial M_g$. We call $\int(\alpha):=\a((0,1))$ the \textit{interior} of the arc, and $\a(0)$ and $\a(1)$ the \textit{endpoints}. We say an arc is \textit{simple} if $\a|_{(0,1)}$ is injective. We often blur the difference between an arc and its image. We define the intersection number between any two maps as $i(f,g):=|\Image(f)\cap \Image(g)|$. 

Let $I$ be a collection of open intervals in $\partial A$. A \textit{doubly-tethered disc} is a pair $(D,\a)$, where:

\begin{itemize}
    \item $D$ is an embedded disc, 
    \item $\a$ is a \textit{double tether}, i.e. a simple arc with $\int(\a)\subset\partial M_g\setminus A$ and both endpoints on $I$, 
    \item $i(\a,D)=1$.
\end{itemize}

The \textit{doubly-tethered disc complex} $\DTD(M_g,A,I)$ is the simplicial complex whose $k$-simplices are collections of $k+1$ isotopy classes of doubly-tethered discs with a set of pairwise disjoint representatives $\{(D_i,\a_i)\}_{0\leq i\leq k}$  such that $M_g\setminus\cup D_i$ is connected. In subsection \ref{Subsec3.6}, we will prove:

\begin{theor}\label{Teo-dtd}
$\DTD(M_g,A,I)$ is $\lfloor(g-3)/2\rfloor$-connected.
\end{theor}

As represented in Diagram \ref{diagr1}, the proof of Theorem \ref{Teo-dtd} cannot be directly deduced from Lemma \ref{DiscConect}; we first need to introduce an extended version of $\DTD(M_g,A,I)$. This is needed as the proof of Theorem \ref{Teo-dtd} relies on Lemma \ref{m-lema}, which we introduce in the following subsection.

\subsection{The extended doubly-tethered disc complex.} 

The \textit{extended doubly-tethered disc complex} $\DTD^+(M_g,A,I)$ is the simplicial complex whose $k$-simplices are collections of $k+1$ isotopy classes of doubly-tethered discs with a set of representatives $\{(D_i,\a_i)\}_{0\leq i\leq k}$ such that:

\begin{itemize}
    \item any two $D_i$'s are either identical or disjoint,
    \item $M_g\setminus\cup D_i$ is connected,
    \item if $D_i\neq D_j$ then $\a_i\cap D_j=\emptyset$.
    \item if $\a_i\cap \a_j\neq \emptyset$, then $\a_i \cap  T_{D_j}^k(\a_j)=\emptyset$ for some $k\in\ZZ$, where $T_{D_j}$ is the (positive) Dehn twist around $\partial D_j$.
\end{itemize}

Note that the complex $\DTD^+(M_g,A,I)$ has the same set of vertices as $\DTD(M_g,A,I)$, but we relax the condition of disjointness of the double tethers in the sense that we allow intersections between the doubly-tethered discs of a simplex as long as those intersections can be ``untwisted'': If $(D,\a)$ and $(D,\b)$ span a 1-simplex, then so does any pair $\{(D,\a),(D,T_D^k(\b))\}$ (see Figure \ref{double tethered +}). 

\begin{figure}[h]
 \includegraphics[width=10cm]{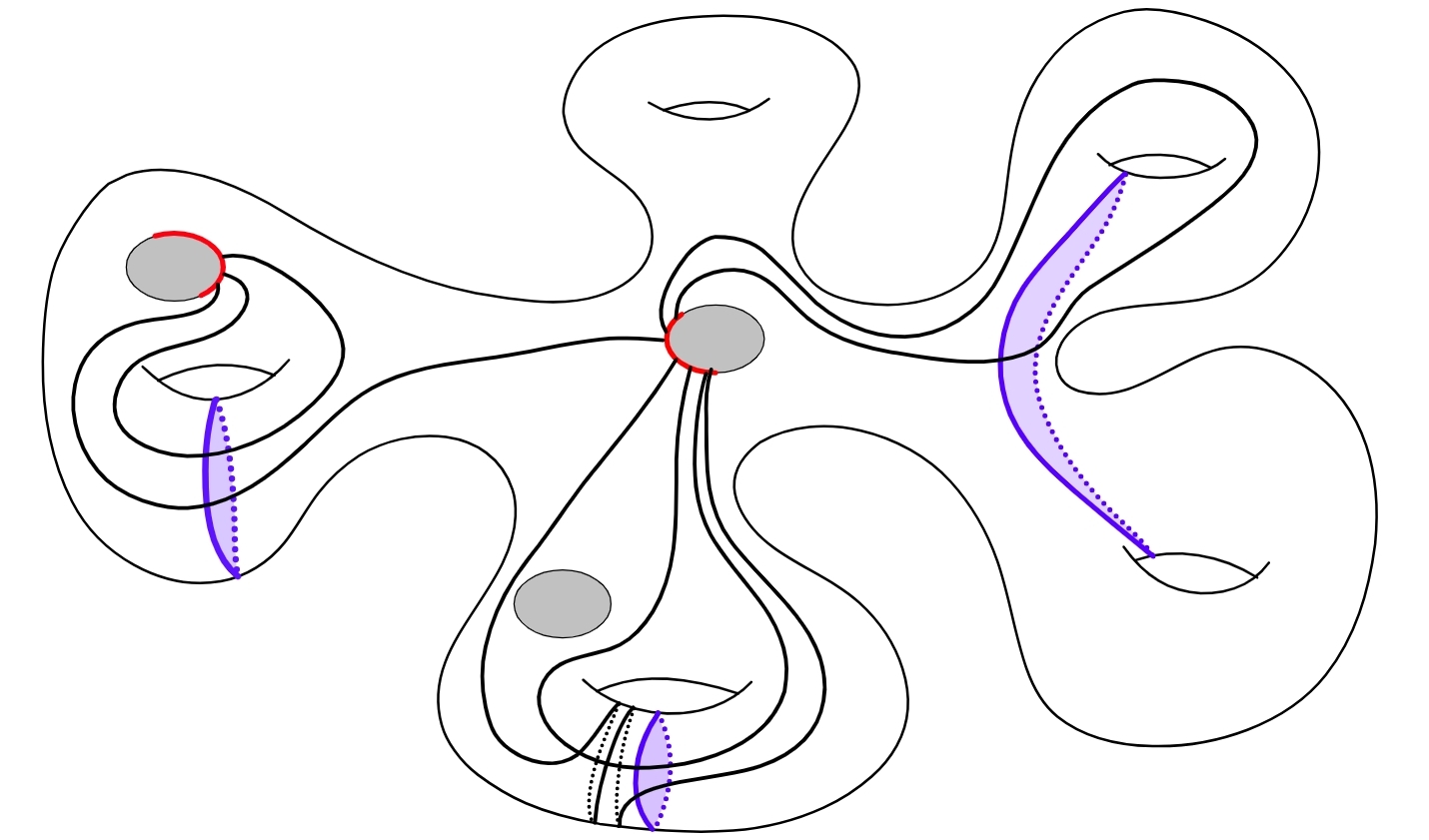}
\centering
\caption{The doubly-tethered discs in the picture form a 4-simplex. Every intersection between the double tethers happens in a regular neighborhood of a disc.}
\label{double tethered +}
\end{figure}

The proof of Theorem \ref{Teo-dtd} requires a result on $\DTD^+(M_g,A,I)$:

\begin{lem}\label{m-lema}
    $\DTD^+(M_g,A,I)_m$ is $(g-m-1)$-connected.
\end{lem}

A closely related version of $\DTD^+(M_g,A,I)$ is introduced in \cite[Subsection 8.3]{Disk}. Our method is heavily inspired by that of \cite{Disk}, but instead of requiring double tethers to be attached at a fixed point $x_0$, we allow attachment at any point in $I$. Given a simplicial complex $X$, denote by $\Star(x,X):={x}*\Lk(x,X)$ the \textit{star} of $x$ in $X$. 

\begin{proof}[\sc Proof of Lemma {\rm \ref{m-lema}}]

Define a simplicial map

$$p: \DTD^+(M_g,A,I) \rightarrow \D_c(M_g,A),$$

given by forgetting the double tethers of each disc. This induces a simplicial map $$p_m: \DTD^+(M_g,A,I)_m \rightarrow \D_c(M_g,A)_m.$$ We will apply Theorem \ref{fiber2}, more precisely, we will prove the contractibility of $p_m^{-1}(v)$ for each $v\in \D_c(M_g,A)_m$, so we can deduce that $\DTD^+(M_g,A,I)_m$ is $(g-m-1)$-connected from Lemma \ref{DiscConect} and Theorem \ref{teo-m}. 

Let $v$ a vertex in $\D_c(M_g,A)_m$, so $v=[D_0,D_1,...D_{k}]$, with $k\geq m- 1$. For each $\partial D_i$ fix a regular neighborhood $R_i$ in $\partial M_g$ so that $R_{j_1}\cap R_{j_2}=\emptyset$ if $j_1\neq j_2$ and $R_j\cap A=\emptyset$ for $ 0\leq j\leq k$. Set $\R_k:=\underset{{0\leq i\leq k}}{\cup}R_i$. Fix a double tether $\alpha_i$ for each $D_i$ so that $\a_{j_1}\cap \a_{j_2}=\emptyset$ if $j_1\neq j_2$ and such that each $\a_j$ intersects $\R_k$ in a single arc that crosses $R_j$ from one boundary of $\partial R_j$ to
the other. The set of the doubly-tethered discs $\{(D_i,\a_i)\}_{0\leq i\leq k}$ forms a vertex $\Bar{v}\in\DTD^+(M_g,A,I)_m $ with $\Bar{v}\in p_m^{-1}(v)$. Set $\A_k:=\underset{{0\leq i\leq k}}{\cup}\a_i$.

We think of any $r$-simplex $\tau\in p_m^{-1}(v)$ as a chain $[\{ \b_0,...,\b_{l_0}\}\leq\{ \b_0,...,\b_{l_0},...,\b_{l_1}\}\leq\{\b_0,...,\b_{l_1},...,\b_{l_2}\}\leq...\leq\{\b_0,...,\b_{l_{r-1}},...,\b_{l_r}\}]$ of isotopy classes of arcs with both endpoints in $I$, and each intersecting exactly one disc in $\{D_0,D_1,...D_{k}\}$ (because $\tau\in p_m^{-1}(v)$). Moreover, we choose representatives for each of the $\b_j$ that minimize the intersections with $\A_k$ outside of $\R_k$ while intersecting $\R_k$ in a single arc that crosses some $R_t$ from one boundary of $\partial R_t$ to
the other. This minimal position is unique up to isotopy.

This gives a description of $p_m^{-1}(v)$ that we will use to prove its contractibility. Let $Y$ be the subcomplex of $p_m^{-1}(v)$ spanned by vertices $\{ \b_0,...,\b_{l}\}$ such that each $\b_j$ either is disjoint from $\A_k$ outside of $\R_k$, or $\b_j=\a_i$ for some $0\leq i\leq k$.

\textit{Claim 1:} $Y$ is contractible.

To each $w=\{ \b_0,...,\b_{l}\}\in Y$ we can assign a vertex $w'=\{\b_1,...,\b_{l_w},\a_1,...,\a_k\}\in  \Star(\Bar{v},p_m^{-1}(v))$. We can retract $Y$ onto $\Star(\Bar{v},p_m^{-1}(v))$ by varying the weight from each $w$ to its corresponding $w'$ (in barycentric coordinates). As the space $\Star(\Bar{v},p_m^{-1}(v))$ is contractible, we can deduce that $Y$ is contractible

\textit{Claim 2:} There is a continuous retraction from $p_m^{-1}(v)$ onto $Y$. 

We will apply Theorem \ref{teoflow}. We first introduce the objects involved, and then check the hypotheses of the theorem. Define $n(\b_i,\a_j):=|\b_i\cap \a_j\setminus R_j|$. Define a complexity function $$c:p_m^{-1}(v)\rightarrow \ZZ,$$ $$c(w)=\underset{i,j}{\sum}n(\b_j,\a_i).$$

The complex spanned by vertices with $c(w)=0$ is $Y$. We now explain, for a given $w\in p_m^{-1}(v)\setminus Y$, how to choose $\Delta w$. Let $w=\{ \b_0,...,\b_{l}\}$. Take the first $i$ such that $\a_i$ intersects one of the $ \b_0,...,\b_{l}$ outside of the $R_i$. Such $i$ always exists, otherwise $w\in Y$. Choose the first $j$ such that $\a_i$ intersects $\b_j$ from its beginning but before $R_i$ (or from its end, if there is no such intersection). Define $\b_j'$ to be the arc going parallel to $\a_i$ from its beginning (resp. end)  until its first intersection with $\b_j$, and then going parallel to $\b_j$ and intersecting a disc (only one of the halves of $\b_j$ intersects a disc) (See Figure \ref{Unicorn}). This procedure is similar to the way to obtain \textit{unicorn arcs} in \cite[Section 3]{Unicorn}.

We choose $\Delta w$ as follows:

\begin{enumerate}
    \item If $\b_j'\notin w$, take $\Delta w=w\cup\{\b_j'\}$.
    
    \item If $\b_j'\in w$, take $\Delta w=w\setminus\{\b_j\}$.
\end{enumerate}

\begin{figure}[ht]
\includegraphics[width=7cm]{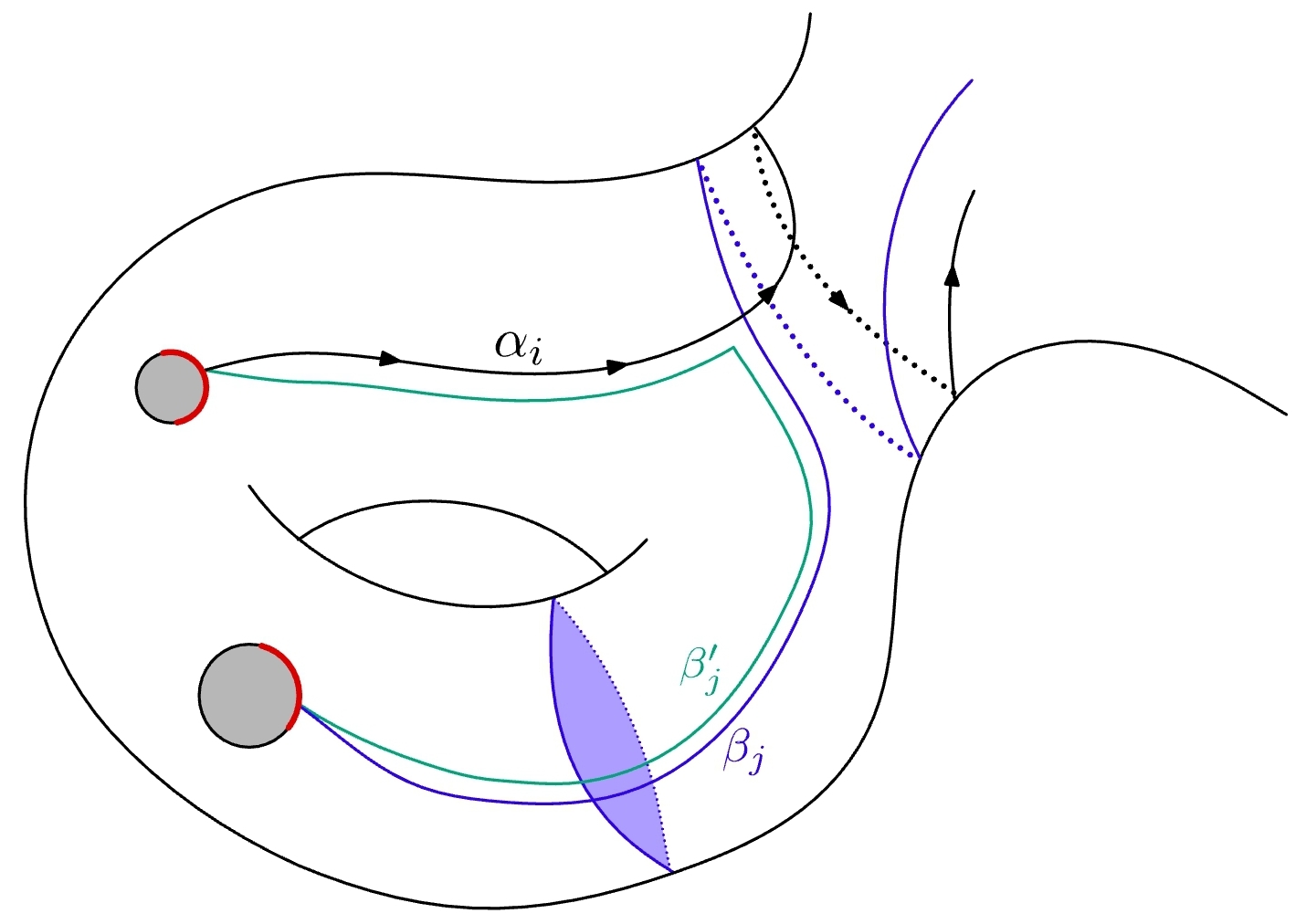}
\centering
\caption{$\b_j'$ intersects $\a_i$ at least once less than $\b_j$.  }
\label{Unicorn}
\end{figure}

Let $\tau=[\{ \b_0,...,\b_{l_0}\}\leq\{ \b_0,...,\b_{l_0},...,\b_{l_1}\}\leq...\leq\{\b_0,...,\b_{l_{r-1}},...,\b_{l_r}\}]\in p_m^{-1}(v)$. We choose $w_\tau$ as follows: Choose $\b_j\in \{\b_0,...,\b_{l_{r-1}},...,\b_{l_r}\} $ and $\b_j'$ as in the definition of $\Delta w$.

\begin{enumerate}
    \item If there is a vertex in the chain $\tau$ where $\b_j$ belongs but $\b_j'$ does not, take $w_\tau$ to be the highest such vertex in the chain $\tau$.
    
    \item If there is no such vertex in the chain $\tau$, take $w_\tau$ to be the lowest vertex in the chain containing both $\b_j$ and $\b_j'$.
\end{enumerate}

\begin{figure}[H]
\begin{center}
\begin{tabular}{ c c c c c c c c c}
 $C\cup \{\b_j\}$ &  & $C\cup \{\b_j, \b_j'\}$ &  & $C\cup \{\b_j, \b_j'\}$ &  & $C\cup \{\b_j, \b_j'\}$ & & $C\cup \{\b_j'\}$\\   
 $B\cup \{\b_j\}$ & $\rightarrow$ & $B\cup \{\b_j\}$ & $\rightarrow$ & $B\cup \{\b_j, \b_j'\}$ & $\rightarrow$ & $B\cup \{\b_j'\}$ & $\rightarrow$ & $B\cup \{\b_j'\}$\\ 
 $A$ &  & $A$ &  & $A$ &  & $A$ & &$A$ 
\end{tabular}
\end{center}
\caption{Each column represent a simplex in $p_m^{-1}(v)$ where each set of arcs is contained in the one immediately over it. $A$, $B$ and $C$ are collections of arcs not containing $\b_j$ or $\b_j'$. The flow takes the left 2-simplex to the right 2-simplex passing through several other chains of collections of  arcs.}
\label{flowfigure}
\end{figure}

This form of choosing $\Delta w$ and $w_\tau$ satisfies conditions (1) and (3) in Theorem \ref{teoflow}. After a certain finite number of steps, every arc $\b_j$ is either eliminated or substituted by $\b_j'$. So after a finite number of steps, image by $c$ of the new simplices will be strictly lower (see Figure \ref{flowfigure}), and thus condition (2) is also satisfied and Theorem \ref{teoflow} can be applied.

Because of Theorem \ref{teo-m} and Lemma \ref{DiscConect} we know $\D_c(M_g,A)_m$ is $(g-m-1)$-connected. And $\DTD^+(M_g,A,I)_m$ is $(g-m-1)$-connected because of Theorem \ref{fiber2}.

\end{proof}

\subsection{Proof of Theorem \ref{Teo-dtd}}\label{Subsec3.6}

We begin the subsection by introducing the necessary tools for the proof. Given a simplicial complex $X$, we denote by $|X|$ the topological space given by the union of all its simplices. A \textit{triangulation} of a manifold $M$ (possibly with boundary) is a simplicial complex $X$ such that $|X|=M$. The following lemma can be found in \cite[Lemma 3.1]{Disk}.

\begin{lem}\label{coloring}(Coloring lemma). Let a triangulation of $\SS^k$ be given with its vertices labeled by elements of a set $E$ having at least $k+2$ elements. Then this labeled triangulation can be extended to a labeled triangulation of $\DD^{k+1}$ whose only simplices with multiple vertices with the same label lie in $\SS^k$, and with the triangulation of $\SS^k$ as a full subcomplex. The labels on the interior vertices of $\DD^{k+1}$ can be chosen to lie in any subset $E_0\subset E$ with at least $k+2$ elements.
\end{lem}

Given two simplicial complexes $X$ and $Y$, we say $Y$ is a \textit{subdivision} of $X$ if $|X|=|Y|$ and every simplex $\sigma\in Y$ satisfies $\sigma\subset\tau$ for some simplex $\tau\in X$. Now we are ready to prove that $\DTD(M_g,A,I)$ is $\lfloor(g-3)/2\rfloor$-connected.

\begin{proof}[\sc Proof of Theorem {\rm \ref{Teo-dtd}}]

For any $k\leq (g-3)/2$, we will prove that any continuous map $$f:\SS^k\rightarrow 
\DTD(M_g,A,I),$$ can be extended to a continuous map $\hat{f}: \DD^{k+1}\rightarrow \DTD(M_g,A,I)$, thus proving that every $\pi_k$ for $k\leq (g-3)/2$ is trivial.

By the Simplicial Approximation Theorem (see \cite[Section 3.4]{Aprox}), we can find a triangulation $X_0$ of $\SS^k$, and a simplicial map $g:X_0\rightarrow 
\DTD(M_g,A,I)$ homotopic to $f$. The strategy now is as follows. Firstly we will promote the map $g$ to a simplicial map $h:X_1\rightarrow 
\DTD^+_{k+2}(M_g,A,I)$, where $X_1$ is some subdivision of $X_0$ yet to be defined. Because of Lemma \ref{m-lema}, this map can be extended to $\hat{h}:\hat{X_1}\rightarrow 
\DTD^+_{k+2}(M_g,A,I)$ for $\hat{X_1}$ a triangulation of $\DD^{k+1}$. Secondly, we will recover a map $\hat{g}:\hat{X_2}\rightarrow \DTD(M_g,A,I)$ with $\hat{X_2}$ a subdivision of $\hat{X_1}$ and such that $\hat{g}|_{\SS^k}$ is homotopic to $f$. This will involve the use of Lemma \ref{coloring}. The homotopy between $f$ and $\hat{g}|_{\SS^k}$ will then be used to build a continuous extension $\hat{f}$ of $f$.

\textit{Step 1: Construction of $h:X_1\rightarrow 
\DTD^+_{k+2}(M_g,A,I)$}. Let $X_0'$ be the barycentric subdivision of $X_0$. We will find a subdivision $X_1$ of $X_0'$ and a simplicial map $h:X_1\rightarrow \DTD^+_{k+2}(M_g,A,I)$ so that the following property is satisfied: 

\begin{quote}
    \textit{Property 1:} For any vertex $v$ in $X_1$ such that $v$ is in the interior of a simplex $[\sigma_0\leq\sigma_1\leq...\leq\sigma_l]\in X_0'$, $h(v)$ has $g(\sigma_0)$ as a \textit{monic} subset, that is, every doubly-tethered disc in $g(\sigma_0)$ is in $h(v)$, and each disc in $g(\sigma_0)$ is intersected by exactly one tether in $h(v)$.
\end{quote}

In other words, discs of $g(\sigma_0)$ are not \say{repeated} in $h(v)$ (see Figure \ref{monic-set}).

\begin{figure}[ht]
\includegraphics[width=10cm]{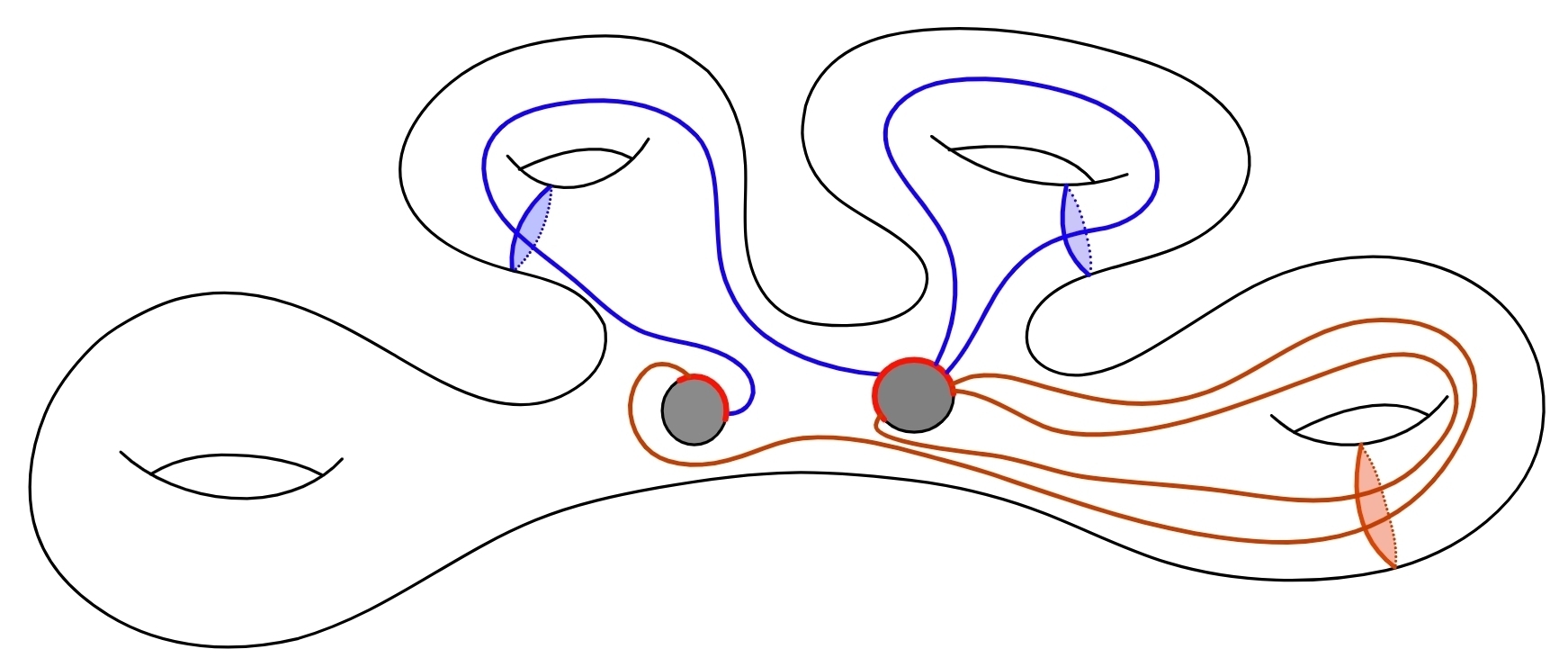}
\centering
\caption{Blue double-tethered discs are a monic subset of the simplex at display. Red are not.}
\label{monic-set}
\end{figure}

We will define $h$ with Property 1 by an inductive argument in $l$: each step defines $h$ on the $l$-skeleton of $X_0'$. We can define $h$ on the $0$-skeleton of $X_0'$: for each vertex $\sigma\in X_0$, we complete $g(\sigma)$ to a simplex $g(\sigma)^+$ of dimension $k+1$, which can be done as every simplex in $\DTD(M_g,A,I)$ can be completed to a maximal $(g-1)$-simplex, and $k+1<g-1$, and we set $h(\sigma)$ to be the vertex in $\DTD^+_{k+2}(M_g,A,I)$ corresponding to the simplex $g(\sigma)^+$. This will serve as our base case.

Now we proceed with the induction step. We define a simplicial extension of $h$ on a subdivision $X_1^l$ of the $l$-skeleton of $X_0'$, provided it has already been defined on a subdivision $X_1^{l-1}$ of its $(l-1)$-skeleton. The process ends with $X_1:=X_1^{k}$, as there are no higher dimension simplices. 

Suppose that there exists a map $h$ satisfying Property 1 defined on a subdivision $X_1^{l-1}$ of the $(l-1)$-skeleton of $X_0'$. For any $l$-simplex $\Sigma=[\sigma_0<\sigma_1<...<\sigma_l]$ of $X_0'$, since $h$ is already defined on $\partial\Sigma$, consider $h_{|\partial\Sigma}:\partial\Sigma\rightarrow \DTD^+_{k+2}(M_g,A,I)$. Property 1 is satisfied by $h_{|\partial\Sigma}$, so every vertex in $h(\partial\Sigma)$ has $g(\sigma_0)$ as a monic subset. We will now follow a surgery argument to find an extension of $h$ in a subdivision $X_1^l$ of the $l$-skeleton of $X_0'$, that still satisfies Property 1. Let $M_{g-q}$ be the spotted handlebody obtained after cutting $M_g$ along the $q$ discs in $g(\sigma_0)$, and let $A'$ be the union of $A$ and the new spots from the surgery. The map $h$ induces a map $h':\partial\Sigma\rightarrow \DTD^+_{k+2-q}(M_{g-q},A',I)$ by setting $h'(\tau)$ for each $\tau\in\partial\Sigma$ to be the set of doubly-tethered discs in $h(\tau)\setminus g(\sigma_0)$. Because of Lemma \ref{m-lema}, we know that $\DTD^+_{k+2-q}(M_{g-q},A',I)$ is $(g-k-3)$-connected. We know:  

\begin{enumerate}

\item$k\leq \frac{g-3}{2}$, hence $k\leq g-3-k$,

\item $l\leq k$, thus $l-1\leq g-3-k$,

\end{enumerate}

so we can find a simplicial extension $h':Y_{\Sigma}\rightarrow \DTD^+_{k+2-q}(M_{g-q},A',I)$ for some triangulation $Y_{\Sigma}$ of $\Sigma$ which agrees with $X_1^{l-1}$ in $\partial \Sigma$. 

We now want to recover a simplicial extension $h:Y_{\Sigma}\rightarrow \DTD^+_{k+2}(M_{g},A,I)$ of $h$. The only possible impediment is that the doubly-tethered discs in $h'(v)$ might not be disjoint from tethers in $g(\sigma_0)$. But this can be easily solved by rerouting the new tethers around the discs in $g(\sigma_0)$ (see Figure \ref{rerouting}). We always reroute the tether with the nearer intersection to the embedded disc, which does not create new intersections, so the process eventually ends. After the rerouting process, the map $h'$ can then be used to define $h$ on any simplex $\tau\in Y_\Sigma$ by setting $h(\tau)=h'(\tau)\cup g(\sigma_0)$. After a finite number of steps (as $X_0'$ is finite) we have defined $h$ on every simplex of the $l$-skeleton of $X_0'$, and $h$ is simplicial as a map $h:X_1^l\rightarrow 
\DTD^+_{k+2}(M_g,A,I)$, with $X_1^l=X_1^{l-1}\underset{l\text{-simplex. }\hspace{3pt}}{\underset{\hspace{3pt}\Sigma\in X_0'}\bigcup Y_\Sigma}$. As $h|_{Y_\Sigma}$ matches $h|_{X_1^{l-1}}$, the map $h$ is continuous. This completes the induction step.

\begin{figure}[ht]
\includegraphics[width=13cm]{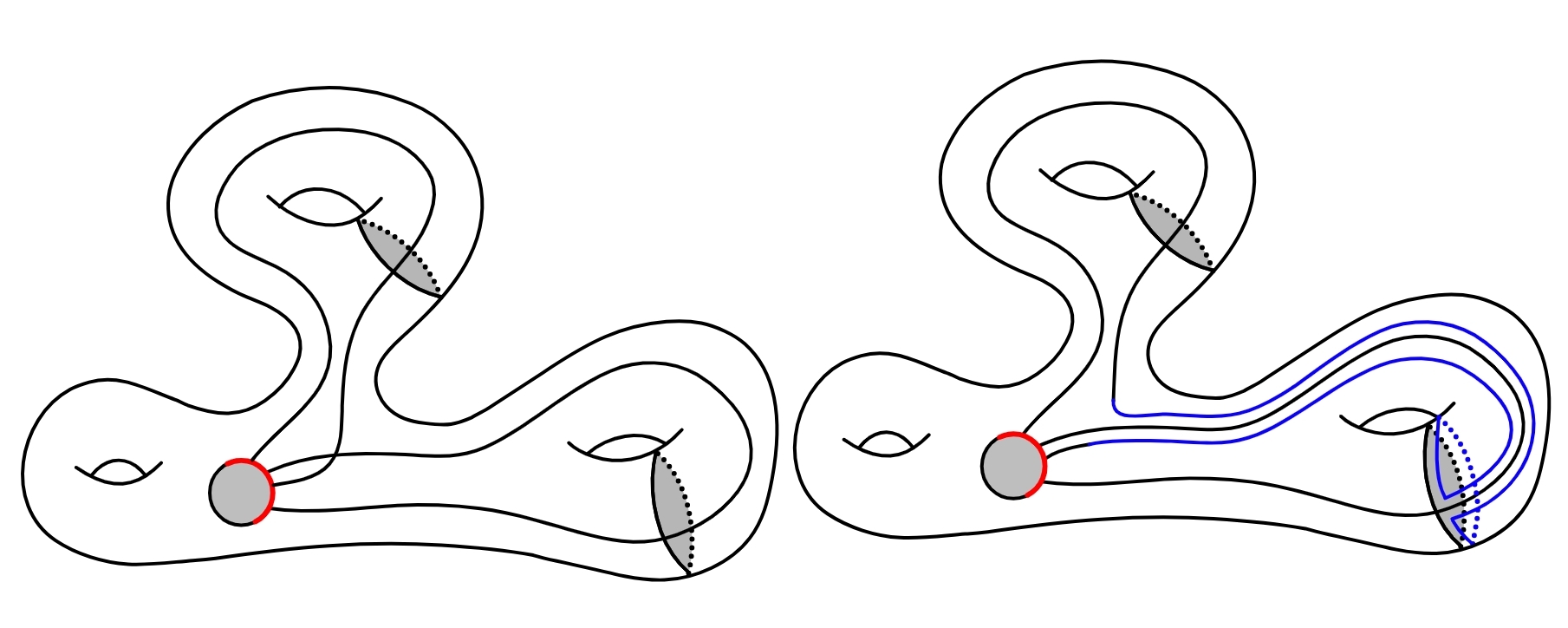}
\centering
\caption{We can avoid the intersection rerouting as shown in the right picture.}
\label{rerouting}
\end{figure}

We have successfully built a continuous simplicial map $h:X_1\rightarrow 
\DTD^+_{k+2}(M_g,A,I)$ satisfying Property 1. By Lemma \ref{m-lema}, $\DTD^+_{k+2}(M_g,A,I)$ is $(g-(k+3))$-connected, and as $k\leq g-(k+3)$, the map $h$ can be extended to a continuous simplicial map $\hat{h}:\hat{X_1}\rightarrow 
\DTD^+_{k+2}(M_g,A,I)$, where $\hat{X_1}$ is a triangulation of $\DD^{k+1}$ with $X_1$ as a full subcomplex.

\textit{Step 2: Construction of $\hat{g}:\DD^{k+1}\rightarrow 
\DTD(M_g,A,I)$.} We will find a subdivision $\hat{X_2}$ of $\hat{X_1}$ and a simplicial map $\hat{g}:\hat{X_2}\rightarrow \DTD(M_g,A,I)$ so that $\hat{g}|_{\SS^k}$ is homotopic to $f$. We first define the map $\hat{g}$ in the vertices $v\in V(\hat{X_1})$: 
  
 \begin{itemize}
     \item If $v\in X_0'$, $\hat{g}(v)$ is any doubly-tethered disc in $g(v)$.
     \item If $v\in X_1\setminus X_0'$, then $v$ lies in the interior of a simplex $[\sigma_0<\sigma_1<...<\sigma_l]\in X_0' $, we set $\hat{g}(v)=\hat{g}(\sigma_0)$.
     \item For any other $v\in \hat{X_1}\setminus X_1$, choose $\hat{g}(v)$ to be any doubly-tether disc in $\hat{h}(v)$.
 \end{itemize}

 \begin{figure}[H]
\includegraphics[width=11cm]{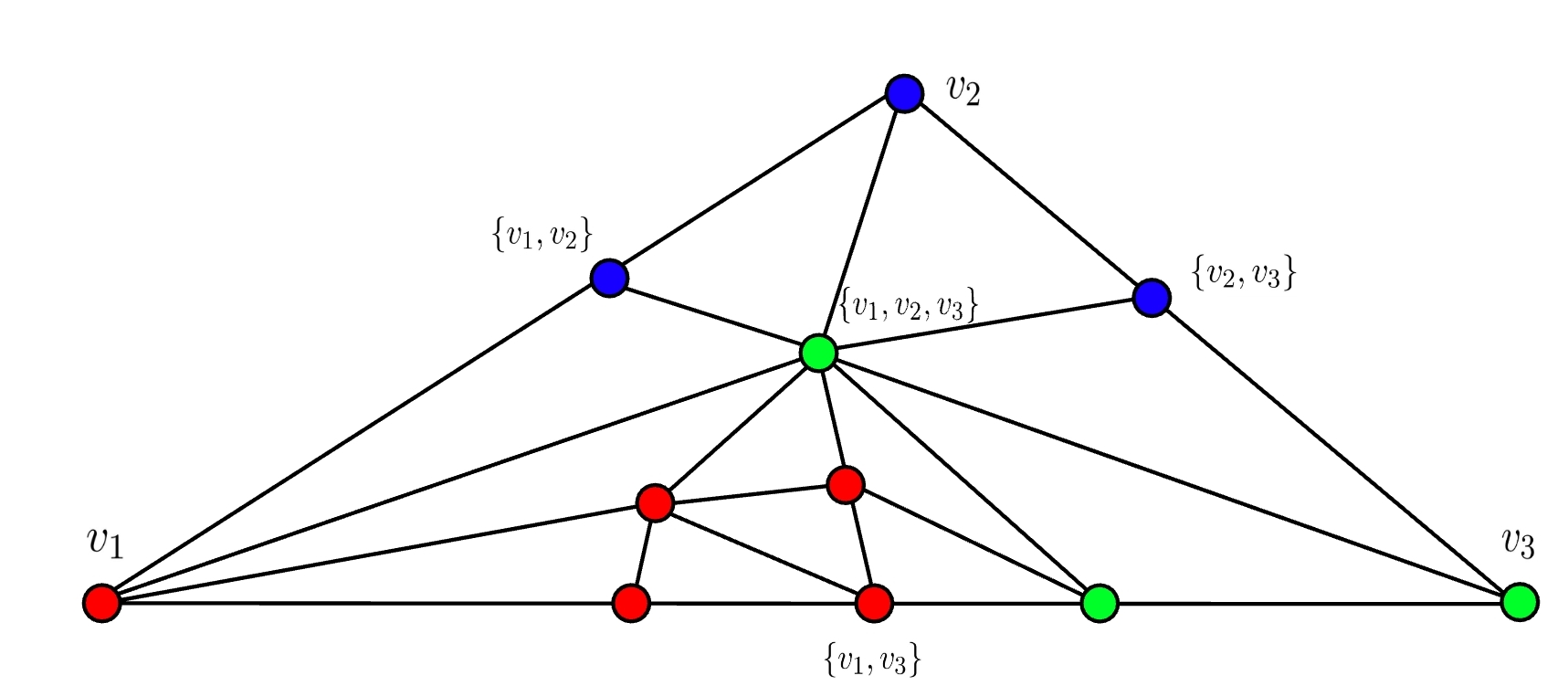}
\centering
\caption{An example of a subdivision of a simplex $\sigma\in X_0$, and a map $\hat{g}$. Colors represent different images under $\hat{g}$.}
\label{Prop2}
\end{figure}
 
The map $\hat{g}$ restricted to $V(X_1)$ can be extended linearly to $X_1$, as vertices sharing an edge have disjoint (or equal) doubly tethered discs as images. For any edge in a simplicial complex, we can define a flow that collapses it (see Figure \ref{flow}), and a sequence of these flows defines an homotopy. The map $\hat{g}|_{S_k}$ is homotopic to $g$: in each simplex $\sigma\in X_0$, we collapse the edges where $\hat{g}$ is constant while leaving the space $\sigma$ invariant, which defines an homotopy between $\Hat{g}|_\sigma$ and $g|_\sigma$.

\begin{figure}[ht]
\includegraphics[width=12cm, height=3cm]{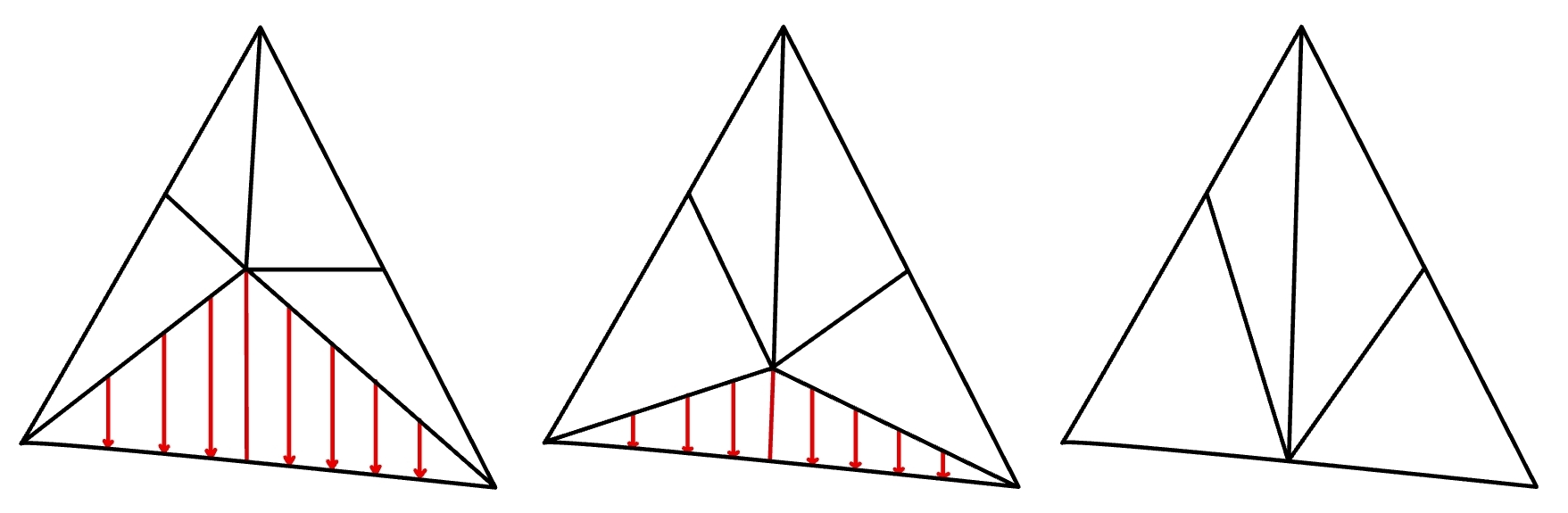}
\centering
\caption{The flow collapses the red edge.}
\label{flow}
\end{figure}

We will define the map $\hat{g}$ in $\hat{X_1}\setminus X_1$ by induction in the dimension of the simplices in $\hat{X_1}$. We already defined it in $V(\hat{X_1})$, this serves as our base case. Analogously to Step 1, we will define $\hat{g}$ to be simplicial on a subdivision $\hat{X_2}^l$ of the $l$-skeleton of $\hat{X_1}$, provided it has already been defined on a subdivision $\hat{X_2}^{l-1}$ of its $l-1$-skeleton.

Define a map $p: \DTD^+(M_g,A,I)\rightarrow \D_c(M_g,A)$ that assigns to each doubly-tethered disc its disc. For any simplex $\tau= [v_0,...,v_l]\in \hat{X_1}$ (with $\hat{h}(v_0)\leq...\leq \hat{h}(v_l)$) of dimension $l$ not contained in $\SS^K$, suppose the map $\hat{g}$ has already been defined on its faces, a collection of simplices of dimension at most $l-1$ in $\hat{X}_2^{l-1}$. Let $E:=p(\hat{h}(v_l))$ and $E_0:=p(\hat{h}(v_0))$. We apply Lemma \ref{coloring} to $\partial \tau$, so we get a triangulation $Y_\tau$ of $\tau$ which agrees with $\hat{X}_2^{l-1}$ in $\partial\tau$. All vertices of the triangulation have colors in $E$, particularly, vertices in $Y_\tau\setminus\partial\tau$ have colors in $E_0$. Any two vertices $v,w\in Y_\tau$ sharing a color satisfy $v,w\in\partial \tau$. 

We define the image $\hat{g}(v)$ of each vertex $v\in Y_\tau$ to be the doubly tethered disc associated to its disc (to its \say{color}) in $\hat{h}(v_0)$. This defines a map $\hat{g}|_{V(Y_\tau)}:V(Y_\tau)\rightarrow \DTD(M_g,A,I)$ (abusing notation). The images of any pair $v,w$ connected by an edge also share an edge, as either they have a different color, and hence their doubly-tethered discs are disjoint, or $v,w\in\hat{X}_2^{l-1}$, where $\hat{g}$ already was defined and simplicial because of the induction hypothesis. So $\hat{g}$ can be extended linearly to the whole of $Y_\tau$. After a finite number of steps we have defined $\hat{g}$ in every simplex of the $l$-skeleton of $\hat{X}_1$, and $\hat{g}$ is simplicial as a map $\hat{g}:\hat{X}_2^l\rightarrow 
\DTD(M_g,A,I)$, with $\hat{X}_2^l=\hat{X}_2^{l-1}\underset{l\text{-simplex.}\hspace{3pt}}{\underset{\hspace{3pt}\tau\in \hat{X}_1}\bigcup Y_\tau}$. This completes the induction step.

We have defined a continuous map $\hat{g}:\DD^{k+1}\rightarrow 
\DTD(M_g,A,I)$ homotopic to $f$ when restricted to $\SS^k$. Realize $\DD^{k+1}$ as the unit ball in $\RR^{k+1}$ centered in the origin, and let $\DD^{k+1}_{1/2}$ be its centered inner disc of radius $1/2$. Take the homotopy between the maps $F:\SS^k\times [1/2,1]\rightarrow \DTD(M_g,A,I) $ with $F|_{\SS^k\times\{1/2\}}=\hat{g}|_{\SS^k}$ and $F|_{\SS^k\times\{1\}}=f$, defined in $\overline{\DD^{k+1}\setminus \DD^{k+1}_{1/2}}$. We define $\hat{f}:\DD^{k+1}\rightarrow\DTD(M_g,A,I)$ as $\hat{f}:=  \begin{cases}
       \hat{g}(2\cdot) &\quad\text{in }\DD^{k+1}_{1/2}\\
       F(\cdot) &\quad\text{in }{\DD^{k+1}\setminus \DD^{k+1}_{1/2}}\\
     \end{cases} .$
This defines a continuous extension of $f$. So for $k\leq (g-3)/2$, every $k$-sphere in $\DTD(M_g,A,I)$ is null homotopic. Therefore, $\DTD(M_g,A,I)$ is $\lfloor(g-3)/2\rfloor$-connected.

\end{proof}

\subsection{The tether-curve-disc complex.}

In \cite{Tethers} Hatcher and Vogtmann prove connectivity results of different complexes of curves and arcs. Their work serves as a guide for the following subsections. Analogously to \cite[Subsection 5.3]{Tethers}, we define the \textit{tether-curve-disc complex}
$\TCD(M_g,A,I)$. 

\begin{figure}[H]
\includegraphics[width=8cm]{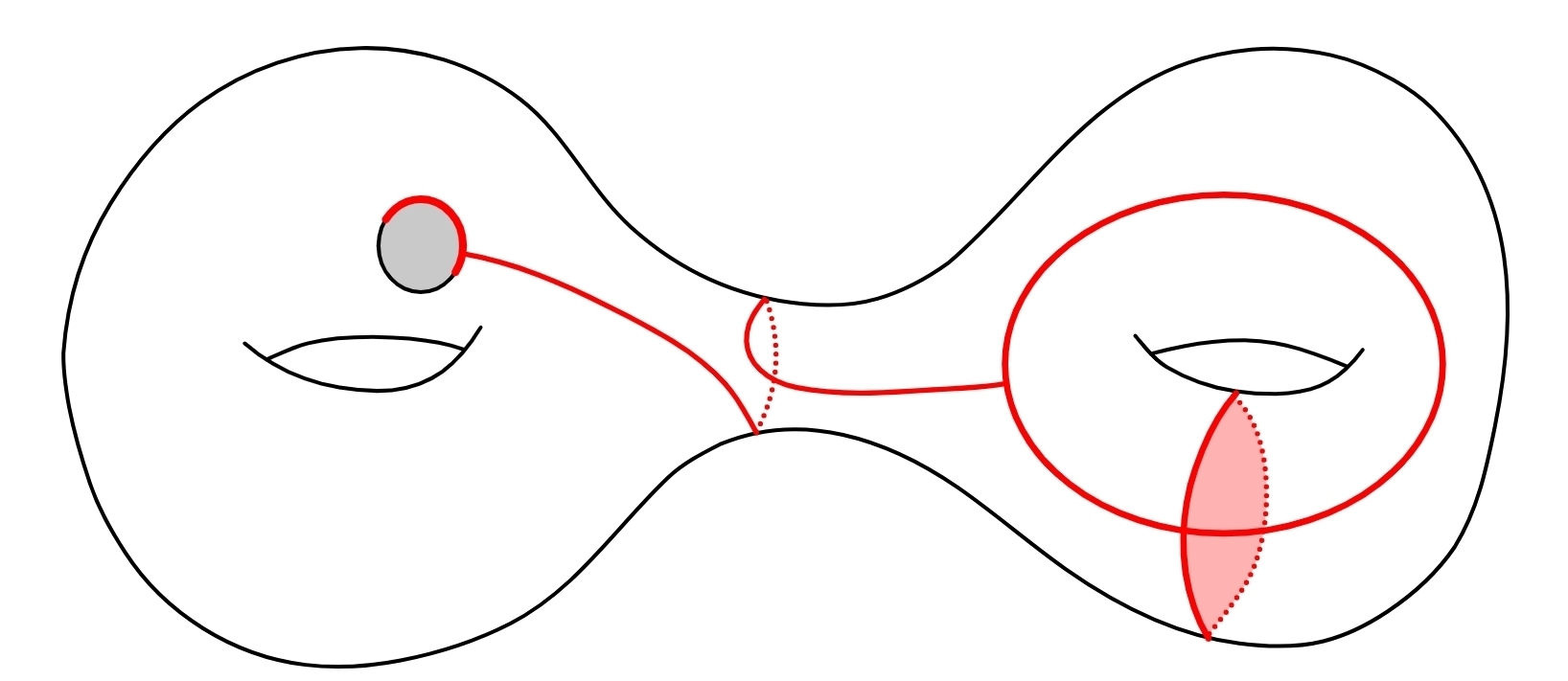}
\centering
\caption{A tether-curve-disc.}
\label{tether-curve-disc}
\end{figure}

Let $(M_g,A)$ be a spotted handlebody. A \textit{curve} $c$ in $\partial M_g$ is a continuous map $c:\SS^1\rightarrow \partial M_g$. We say an curve is \textit{simple} if $c$ is injective. We often blur the difference between a curve and its image.

Let $I$ be a collection of open intervals in $\partial A$. A \textit{tether-curve-disc} (see Figure \ref{tether-curve-disc}) is a triad $(D, c, \a)$ with:
\begin{itemize}
    \item $D$ an embedded disc, 
    \item $c$ a simple curve in $\partial M_g\setminus A$ such that $i(c,D)=1$,
    \item $\a$ a simple arc with $\int(\a)\subset\partial M_g\setminus (A\cup \partial D \cup c)$, an endpoint in $I$, and the other endpoint in $c$. 
\end{itemize}

The tether-curve-disc complex $\TCD(M_g,A,I)$ is the simplicial complex whose $k$-simplices are $k+1$ isotopy classes of tether-curve-discs with a set of pairwise disjoint representatives. Given a tether-curve-disc $(D,c,\a)$, consider a regular neighborhood $N_{\a\cup c}$ of ${\a\cup c}$ in the surface $\partial M_g$. The neighborhood $N_{\a\cup c}$ has two boundary components in $\partial M_g$, one of them, which we denote by $b$, intersecting $I$ at two points. We can define a map $$S: \TCD(M_g,A,I)\rightarrow \DTD(M_g,A,I),$$ 
by taking the arc $\b:=b\setminus \int(A)$ as a double tether for $D$, and declaring $S((D, c, \a))=(D,\b)$. This process is shown in Figure \ref{curve-to-double}. The map $S$ is simplicial, as disjointness is preserved, and injective, as every doubly-tethered disc obtained via the described surgery has exactly one preimage.

\begin{figure}[H]
\includegraphics[width=12.5cm]{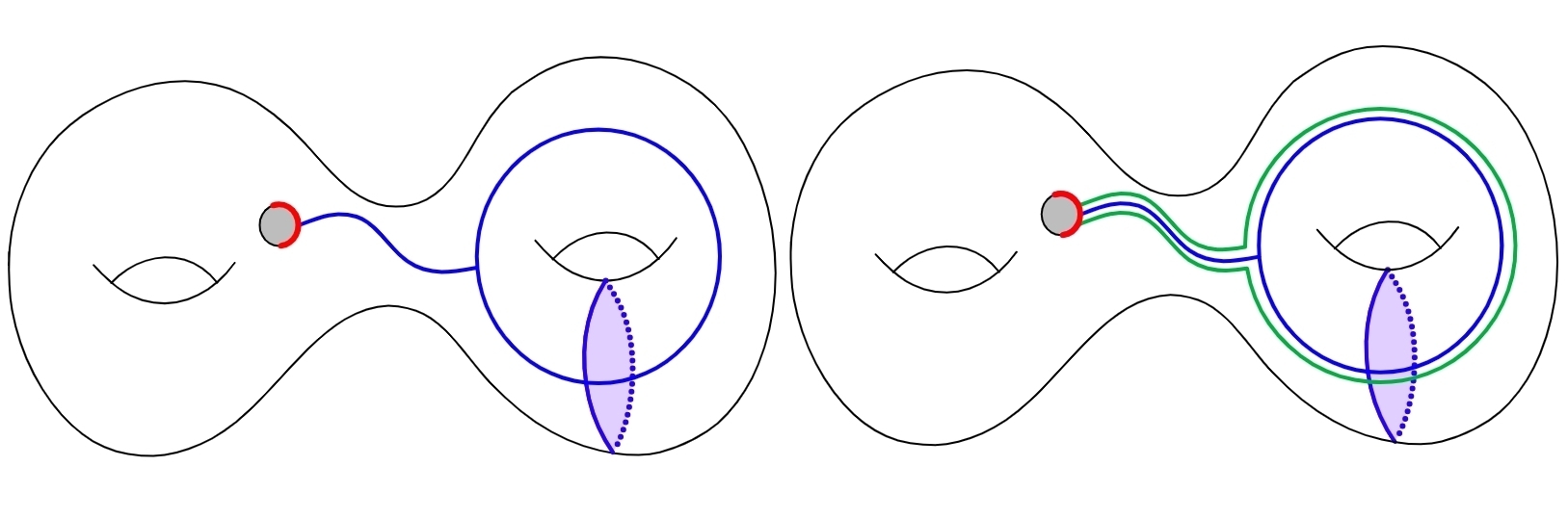}
\centering
\caption{The blue tether-curve-disc gives the green double tether.}
\label{curve-to-double}
\end{figure}

Our following result and its proof are a direct translation of \cite[Proposition 5.5]{Tethers}, hence we only offer an outline of the proof.

\begin{theor}\label{teo-tcd}
The complex $\TCD(M_g,A,I)$ is $\lfloor(g-3)/2\rfloor$-connected.
\end{theor}

\begin{proof}[\sc Outline of the proof of Theorem {\rm \ref{teo-tcd}}] 

The proof uses Theorem \ref{badsim} to prove that $\TCD(M_g,A,I)$ is $\lfloor(g-3)/2\rfloor$-connected, by induction on the genus of $M_g$. When $g=1$, the complex is nonempty, and hence it is $(-1)$-connected. This serves as the base case. We now proceed with the induction step: assume the Theorem to be true for every $g'<g$. 

Assume first $I$ consists of only one open interval in $\partial A$. We use a bad simplex argument. Give an orientation to the interval $I$, and consequently an order to the ends of tethers attached to $I$. Given a simplex $\sigma=[(D_i,\a_i)]_{0\leq i\leq n}\in $ let $(a_i,a_i')$ be the endpoints of each $\a_i$. We set the bad vertices of $\sigma$ to be the $(D_k,\a_k)$ such that $a_k$ and $a_k'$ lie in different components of $I\setminus ( \underset{i\neq k}{\cup}\{a_i,a_i'\})$. The corresponding good complex is precisely $S(\TCD(M_g,A,I))$.

Let $\sigma$ be a bad $k$-simplex. Let $M_{g-(k+1)}$ be the spotted handlebody obtained from $M_{g}$ by cutting through the set of discs in $\sigma$. Let $A':=A\cup R$, where $R$ is a regular neighborhood of both the tethers in $\sigma$ and the new spots which come from cutting through the discs in $\sigma$, and let $I':=I\setminus \Bar{R}$. The good link of $\sigma$ can be seen as the simplicial complex $\TCD(M_{g-(k+1)},A',I')$ (see Figure \ref{pprime}), which is $\lfloor(g-k-4)/2\rfloor$-connected because of the induction hypothesis. We also have $\frac{g-k-4}{2}\geq\frac{g-3}{2}-k$ since $k\geq 1$, as a single vertex cannot be bad. So, according to Theorem \ref{badsim}, the map $S$ induces an isomorphism in $\pi_m$ for every $m\leq\frac{g-3}{2}$, and it follows that $\TCD(M_g,A,I)$ is $\lfloor(g-3)/2\rfloor$-connected.

\begin{figure}[H]
\includegraphics[width=13.5cm]{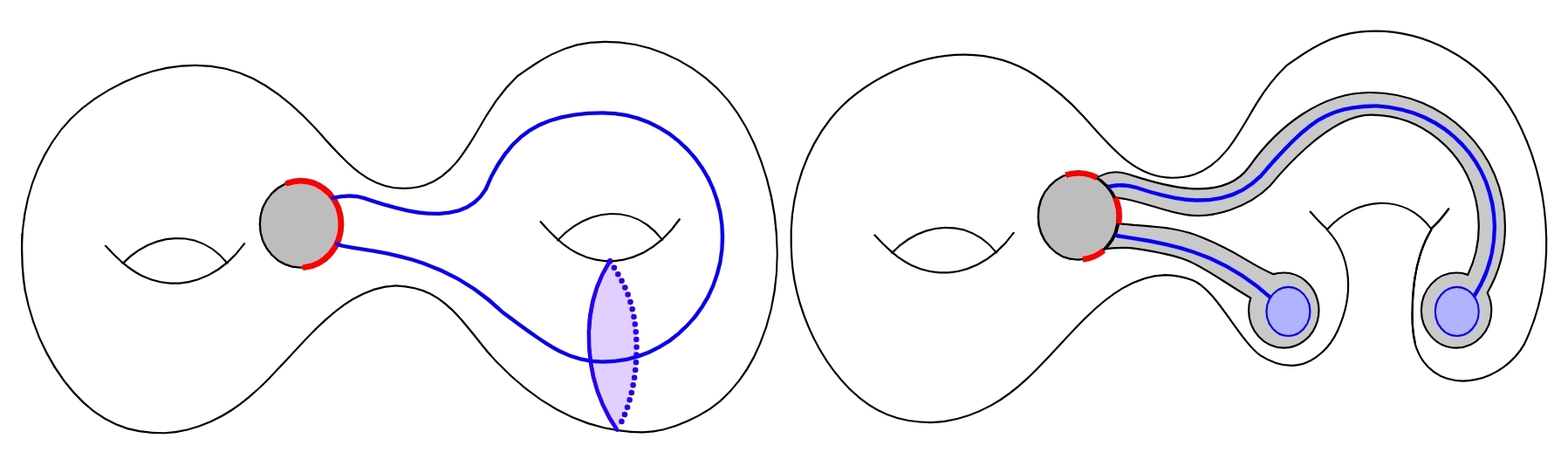}
\centering
\caption{The red segment is one of the segments in $I$ (left picture). In $I'$, it gives us three segments (right picture). All the grey neighborhood of the double tether becomes part of a spot.}
\label{pprime}
\end{figure}

For the case where $I$ consists of several open intervals in $A$ we choose a particular interval $\Bar{I}\subset I$. We have a map $\TCD(M_g,A,\Bar{I})\hookrightarrow \TCD(M_g,A,I)$ given by the inclusion of $\TCD(M_g,A,\Bar{I})$ as a subcomplex. Choosing bad vertices of a simplex to be the tether-curve-discs attaching to $I\setminus \Bar{I}$, a bad simplex argument can be applied for $m=\lfloor(g-3)/2-1\rfloor$, hence proving that $\TCD(M_g,A,I)$ is $\lfloor(g-3)/2\rfloor$-connected for any $I$.
\end{proof}

\subsection{The curve-disc complex.}
The next complex we introduce is our analog to the complex of oriented chains in \cite[Section 5.3]{Tethers}, but before, we need some notation. Every curve $c$ in an orientable surface $S$ has two \textit{sides}, that is, for $N_c$ a regular neighborhood of the curve, two connected components in $N_c \setminus c$. Giving an \textit{orientation} to a curve corresponds to designating one of these sides as \textit{positive} and the other as \textit{negative}. A \textit{curve-disc} is a pair $(D, c)$ with:
\begin{itemize}
    \item $D$ an embedded disc, 
    \item $c$ an oriented simple curve in $\partial M_g\setminus A$ such that $i(c,D)=1$.
\end{itemize}

The\textit{ curve-disc complex} $\CD(M_g,A)$ is the simplicial complex whose $k$-simplices are $k+1$ isotopy classes of curve-discs with a set of pairwise disjoint representatives. Similarly to the case of the doubly-tethered disc complex, the connectivity of $\CD(M_g,A)$ is proven by making use of a modified version of the tether-curve-disc complex. More concretely, we define $\TCD^+(M_g,A)$ to be the complex with same vertex set as $\TCD(M_g,A,I)$, but where we allow several tethers to join to each curve-disc, all attaching to the same point of $c$ and the same side of $c$. We can define a map $$p:\TCD^+(M_g,A,I)\rightarrow\CD(M_g,A)$$ by forgetting the tethers and orienting the curve so that the side the tethers attach to is positive. The following result and its proof are analogous to \cite[Proposition 5.6]{Tethers}. We offer an outline of the proof.

\begin{theor}\label{teo-cd}
The complex $\CD(M_g,A)$ is $\lfloor(g-3)/2\rfloor$-connected.
\end{theor}

\begin{proof}[\sc Outline of the proof of Theorem {\rm \ref{teo-cd}}]
Given $\sigma\in \CD(M_g,A)$ the fiber $p^{-1}(\sigma)$ is an arc complex, with $n$-simplices corresponding to collections of $n+1$ pairwise disjoint isotopy classes of arcs with one endpoint in $I$, and the other in one of the curve-discs in $\sigma$, attached to the positive side of the curve (see Figure \ref{Multicurvedisc}). The fiber $p^{-1}(\sigma)$ can then be proven to be contractible building a flow by using Theorem \ref{teoflow} and a surgery technique similar to the proof of Lemma \ref{m-lema}.

\begin{figure}[H]
\includegraphics[width=10cm]{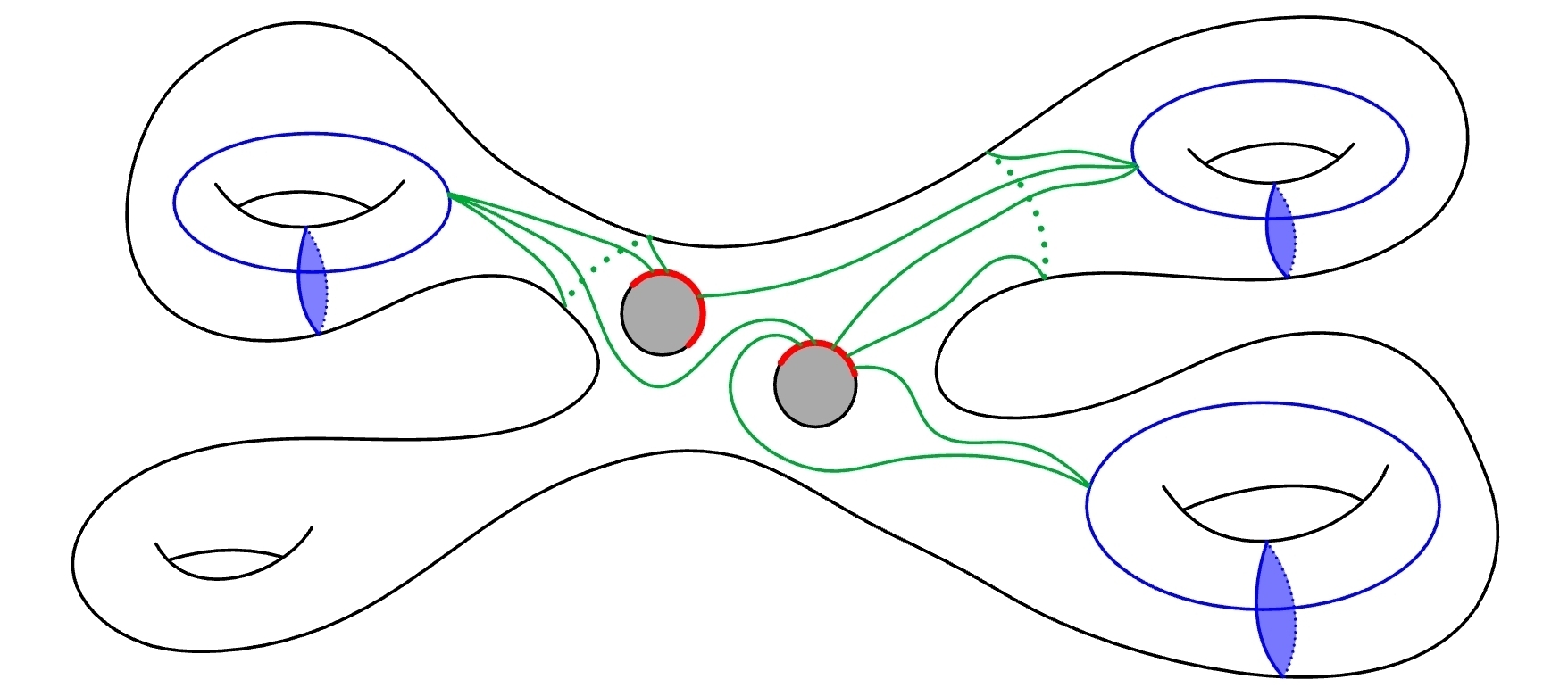}
\centering
\caption{The blue curve-discs form a 2-simplex $\sigma\in \CD(M_g,A)$. The green arcs form a 7-simplex in $p^{-1}(\sigma)$.}
\label{Multicurvedisc}
\end{figure}

Hence, according to Theorem \ref{fiber}, it suffices to check the connectivity of $\TCD^+(M_g,A,I)$. Here a bad simplex argument is used. We have a map $\TCD(M_g,A,I)\hookrightarrow\TCD^+(M_g,A,I)$ given by the inclusion of $\TCD(M_g,A,I)$ as a subcomplex. 
We define the set of bad vertices of a simplex $\sigma\in\TCD^+(M_g,A,I)$ as the set of tether-curve-discs whose curve-discs coincide with at least one other vertex in $\sigma$, and the associated good complex is precisely $\TCD(M_g,A,I)$. Because of Theorem \ref{teo-tcd}, checking the conditions of Theorem \ref{badsim} for $m=\lfloor(g-3)/2\rfloor$ finishes the proof.

\end{proof}

\subsection{The handle complex.} The rest of this section introduces the last group of simplicial complexes in Diagram \ref{diagr1}, and the corresponding results that will be used to prove the connectivity bounds for the descending links. In this last part of the section, we closely follow the reasoning in \cite[Section 9]{Asymp}. 

A \textit{handle} in $M_g$ is a genus one subhandlebody with exactly one spot. The \textit{handle complex} $\Hc(M_g,A)$, is the complex whose $k$-simplices correspond to collections of $(k+1)$ isotopy classes of handles with a set of pairwise disjoint representatives. Observe that a regular neighborhood of a curve-disc is a handle, hence we have a map $$p: \CD(M_g,A)\rightarrow \Hc(M_g,A),$$ which is surjective, as a preimage can always be found by choosing an embedded disc, and an oriented curve (see Figure \ref{curvetohandle}).

\begin{figure}[ht]
\includegraphics[width=7cm]{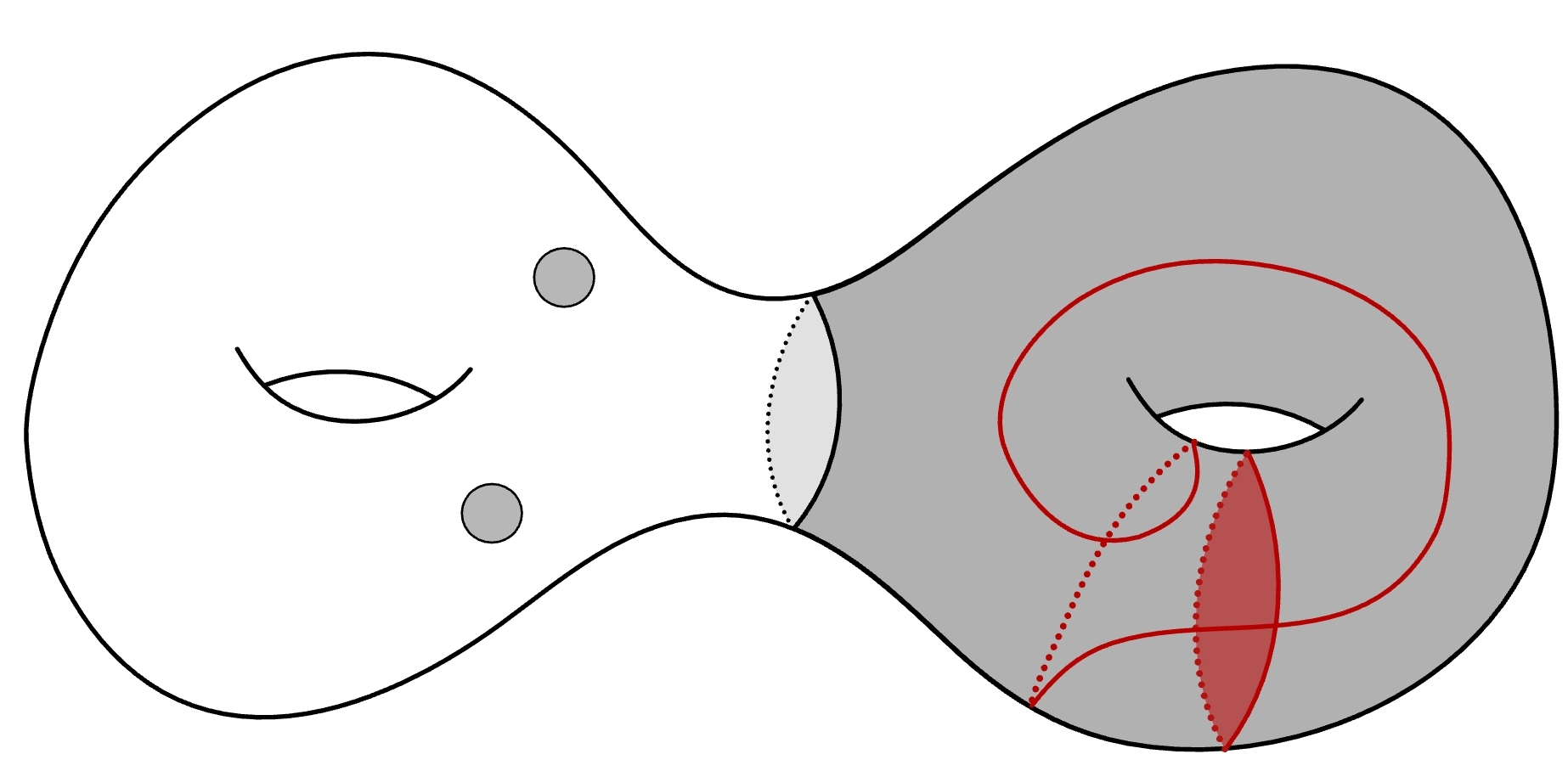}
\centering
\caption{A possible choice of preimages for a handle.}
\label{curvetohandle}
\end{figure}

\begin{lem}\label{teo-joinhandle1}
With respect to the map $p$, the complex $\CD(M_g,A)$ is a complete join over $\Hc(M_g,A)$. Hence $\Hc(M_g,A)$  is $\lfloor(g-3)/2\rfloor$-connected.
\end{lem}

\begin{proof}[\sc Proof of Lemma {\rm \ref{teo-joinhandle1}}]
The second part of the lemma is immediate once the first is proven, because of Theorem \ref{join2} and Theorem \ref{teo-cd}. We now prove the first part. The map $p$ is simplicial, as disjointness is preserved by taking regular neighborhoods, and inyective on simplices, as disjoint curve-discs never share an image.

For any given handle, the disc in the preimage is unique up to isotopy, as there is exactly one disc up to isotopy in a genus one one-spotted handlebody (see Figure \ref{curvetohandle}). On the other hand, the curve is unique up to Dehn twisting around the disc, and choosing an orientation. For any simplex $\sigma= [x_0,...,x_p]\in \Hc(M_g,A)$, each $\hat{x_i}\in p^{-1}(x_i)$ is a curve-disc entirely contained in the handle $x_i$. This implies that each set of preimages $\{\hat{x_0},...,\hat{x_P}\}$ have a set of pairwise disjoint representatives, hence $p^{-1}(\sigma)=p^{-1}(x_0)*p^{-1}(x_1)*...*p^{-1}(x_p)$. According to Definition \ref{def-join}, $\CD(M_g,A)$ is a complete join over $\Hc(M_g,A)$. 

\end{proof}

\subsection{The tethered handle complex.}

Let $(M_g,A)$ be a spotted handlebody and $B\subset A$ a subset of spots. A \textit{tethered handle} consists of a pair $(T,\a)$, with: \begin{itemize}
    \item a handle $T$,
    \item a simple arc $\a\subset\partial M_g$ with $\int(\a)\cap (A\cup T)=\emptyset$, one endpoint in $\partial B$, and the other endpoint in $\partial T$.
    \end{itemize}
The \textit{tethered handle complex} $\TH(M_g,A,B)$, is the complex whose $k$-simplices correspond to collections of $(k+1)$ isotopy classes of tethered handles with a set of pairwise disjoint representatives. As it was the case in previous subsections, we also need a modified version of the tethered handle complex: define $\TH^+(M_g,A,B)$ to be the complex with same vertex set as $\TH(M_g,A,B)$, but where we allow several tethers attached to each handle. We define a map $$p:\TH^+(M_g,A,B)\rightarrow\Hc(M_g,A)$$ by forgetting the tethers.

\begin{figure}[ht]
\includegraphics[width=7cm]{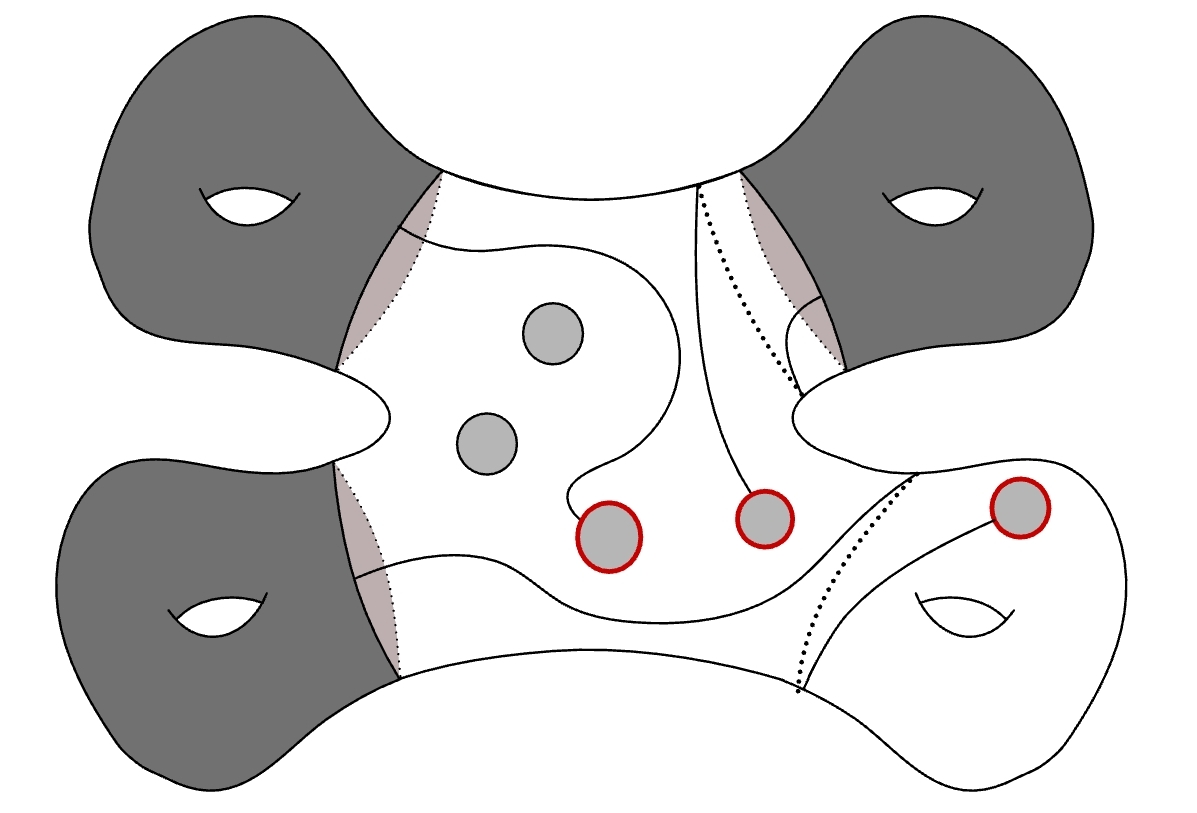}
\centering
\caption{A 2-simplex in $\TH(M_4,A,B)$.}
\end{figure}

The following lemma is analogous to \cite[Lemma 9.12]{Asymp}. We give an outline of the proof. For further details, check the proofs of \cite[Lemma 9.11, Lemma 9.12]{Asymp}.

\begin{lem}\label{teo-th}
$\TH(M_g,A,B)$ is $\lfloor(g-3)/2\rfloor$-connected.
\end{lem}

\begin{proof}[\sc Outline of the proof of Lemma {\rm \ref{teo-th}}]

We will apply Theorem \ref{fiber2} in order to prove that the map $p$ is a homotopy equivalence. The barycentric fiber $p_1^{-1} (\sigma)$ of each $\sigma \in \Hc(M_g,A)$ is isomorphic to an arc complex, with vertices corresponding to collections of pairwise disjoint isotopy classes of arcs with one of its endpoints in $B$, and the other endpoint in one of the handles in $\sigma$; and at least one tether attached to each handle (see Figure \ref{Multihandle}). The arc complexes are contractible, as it is shown in \cite[Proposition A.11]{Asymp}. Hence $\TH^+(M_g,A,B)$ is $\lfloor(g-3)/2\rfloor$-connected.

\begin{figure}[h]
\includegraphics[width=10cm]{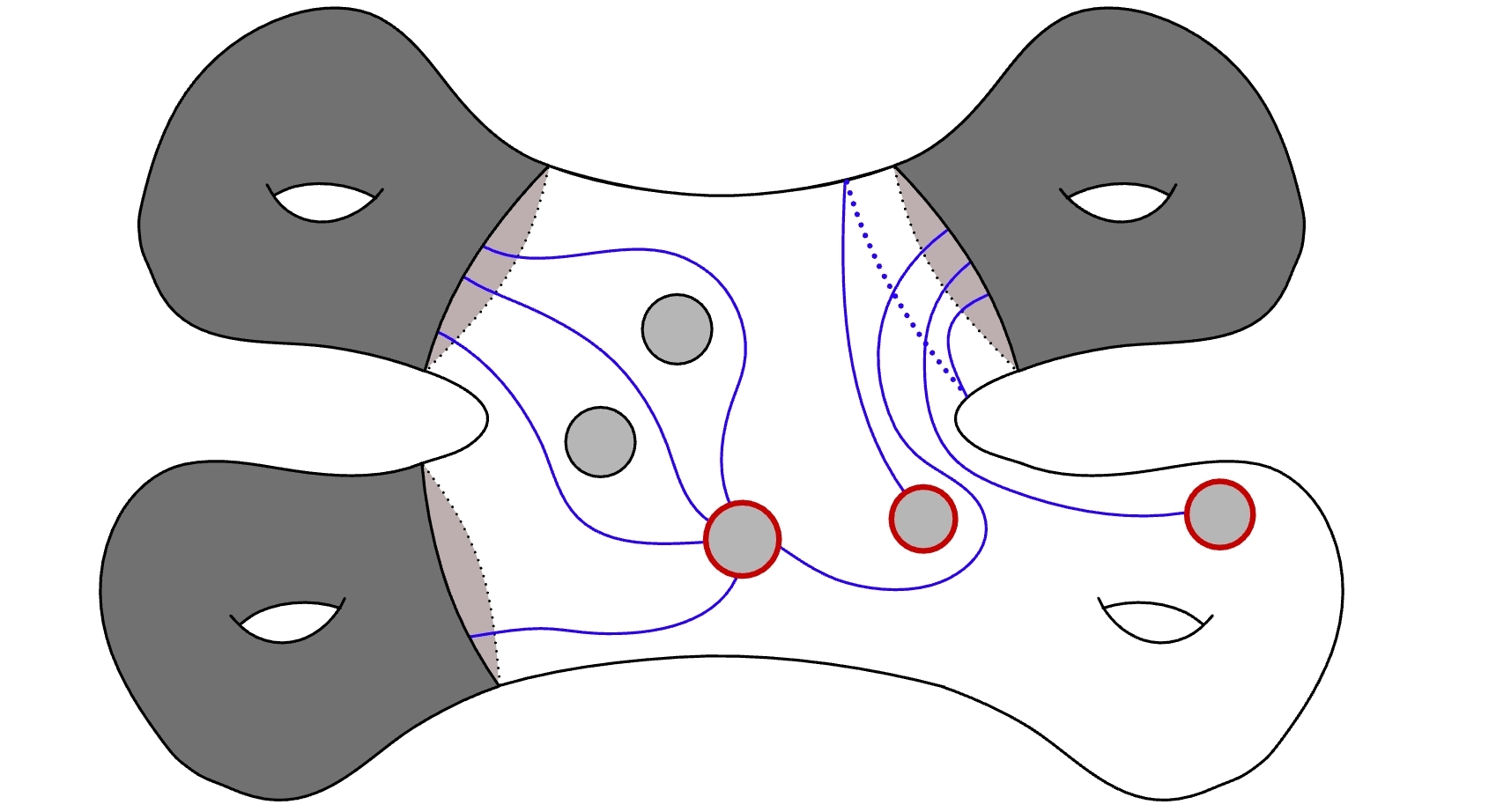}
\centering
\caption{The handles form a 2-simplex $\sigma\in \Hc(M_g,A)$. The blue arcs form a 6-simplex in $p^{-1}(\sigma)$.}
\label{Multihandle}
\end{figure}

Take the map $\TH(M_g,A,B)\hookrightarrow\TH^+(M_g,A,B)$ given by the inclusion of $\TH(M_g,A,B)$ as a subcomplex. We use a bad simplex argument. If we take the set of bad vertices of a simplex $\sigma\in\TH^+(M_g,A,B)$ to be the tethered handles whose handles coincide with at least one other vertex in $\sigma$, the good complex is $\TH(M_g,A,B)$. So checking the conditions of Theorem \ref{badsim} for $m=\lfloor(g-3)/2\rfloor$ finishes the proof.

\end{proof}

\subsection{The outer disc complex.}

The \textit{outer disc complex} $\OD(M_g,A,B)$ is the simplicial complex whose $k$-simplices correspond to collections of $k+1$ isotopy classes of \textit{outer discs}, that is, discs $b\subset\partial M_g$ such that for every spot $s \subset A$ either $s\cap b=\emptyset$ or $s\subset b$, and containing $d$ of the discs in $B$ and no other discs in $A$; with a set of pairwise disjoint representatives. The following result is stated in \cite[Lemma 9.3]{Asymp}, and is a corollary of \cite[Corollary 4.12]{Braided}.

\begin{lem}\label{lema-ball}
    Let $d\geq2$. The complex $\OD(M_g,A,B)$ is $m$-connected for $m\leq \frac{|B|+1}{2d-1}-2$. 
\end{lem}

\begin{rem}
    For $d=1$, the complex $\OD(M_g,A,B)$ is a $(|B|-1)$-simplex, so it is contractible. Hence Lemma 10 remains true if $d=1$.
\end{rem}

\subsection{The extended handle-tether-disc complex.}

A \textit{handle-tether-disc} is a triple $(T,\alpha,b)$ of
\begin{itemize}
    \item a handle $T$,
    \item an outer disc $b$,
    \item a simple arc $\a\subset\partial M_g$ with $\int(\a)$ disjoint from $A$, $b$ and $\partial T$, one endpoint in $\partial b$, and the other endpoint in $\partial T$,  
\end{itemize}

The \textit{extended handle-tether-disc} complex $\HTD^{ +}(M_g,A,B)$ is the simplicial complex whose $k$-simplices are collections of $k+1$ isotopy classes of handle-tether-discs with a set of representatives that pairwise, either are disjoint, or their outer discs are the same and their tethered handle parts ($T$ and $\a$) are disjoint. We can define a simplicial map $$p:\HTD^{+}(M_g,A,B)\rightarrow\OD(M_g,A,B)$$ by forgetting the handle-tethers. The following result and its proof closely mirror \cite[Lemma 9.13]{Asymp}. Hence we only give a brief outline of the proof.

\begin{theor}\label{teo-htb1}

    Let  $m=\min\left\{\left\lfloor\frac{g-3}{2}\right\rfloor,\left\lfloor \frac{|B|+1}{2d-1}-2\right\rfloor\right\}$. Then $\HTD^{+}(M_g,A,B)$ is $m$-connected. 
\end{theor}

\begin{proof}[\sc Outline of the proof of Theorem {\rm \ref{teo-htb1}}]
The main idea is to apply Theorem \ref{fiber} in order to deduce a connectivity bound for $\HTD^{+}(M_g,A,B)$ from Lemma \ref{lema-ball}. For that purpose we check the connectivity of the fibers of $p$.

Given a simplex $\sigma\in \OD(M_g,A,B)$, the fiber $p^{-1}(\sigma)$ can be described as $$p^{-1}(\sigma)=\underset{\tau\subseteq\sigma}\bigcup \TH(M_g,A_\tau,B_\tau),$$ where $B_\tau$ is a set of representatives for the discs in $\tau$, and $A_\tau=A\cup B_\tau$; as any simplex $\theta\in p^{-1}(\sigma)$ belongs to $\TH(M_g,A_{p(\theta)},B_{p(\theta)})$ as a collection of tethered handles.

Under such situation, where the complex $p^{-1}(\sigma)$ is covered by the family of subcomplexes $\{ \TH(M_g,A_\tau,B_\tau)\}_{\tau\subseteq\sigma}$, we can apply a nerve covering argument (see \cite[Lemma 9.13, Proposition A.16]{Asymp}), which applied to our situation yields that $p^{-1}(\sigma)$ is as connected as are the nonempty intersections 
$$\underset{i=1}{\overset{j}{\bigcap}}\TH(M_g,A_{\tau_i},B_{\tau_i})=\TH(M_g,A_{\tau_1\cup...\cup\tau_j},B_{\tau_1\cap...\cap\tau_j}).$$

The connectivity of $\TH(M_g,A_{\tau_1\cup...\cup\tau_j},B_{\tau_1\cap...\cap\tau_j})$ when nonempty is $\lfloor(g-3)/2\rfloor$ according to Lemma \ref{teo-th} and hence, because of Lemma \ref{lema-ball} and Theorem \ref{fiber}, $\HTD^{+}(M_g,A,B)$ is $\min\left\{\left\lfloor\frac{g-3}{2}\right\rfloor,\left\lfloor \frac{|B|+1}{2d-1}-2\right\rfloor\right\}$-connected.

\end{proof}

\subsection{The handle-tether-disc complex.}

With Theorem \ref{teo-htb1} to hand, we can now prove the connectivity of the \textit{handle-tether-disc complex} $\HTD(M_g,A,B)$, a complex with the same set of vertices as $\HTD^{+}(M_g,A,B)$ but where $k$-simplices are collections of $k+1$ isotopy classes of handle-tether-discs with a set of pairwise disjoint representatives. As in other subsections, $\HTD(M_g,A,B)$ can be seen as a subcomplex of its extended version, which gives rise to an inclusion $$i: \HTD(M_g,A,B)\hookrightarrow\HTD^{+}(M_g,A,B)$$

\begin{theor}\label{teo-htb2}
Let $m=\min\left\{\left\lfloor\frac{g-3}{2}\right\rfloor,\left\lfloor \frac{|B|+1}{2d-1}-2\right\rfloor, |B|-2\right\}$. Then $\HTD(M_g,A,B)$ is $m$-connected. 
\end{theor}

\begin{proof}[\sc Proof of Theorem {\rm \ref{teo-htb2}}]
We will prove the result by induction in the complexity $c:=\min\left \{ dg, |B|\right\}$ of the triad $(M_g,A,B)$. We take triads with $c< 2d$, i.e. either satisfying $g= 1$, or $d \leq|B|\leq 2d$, as base cases. The theorem holds true as $m=-1$ in these cases, and the complex is nonempty. We now proceed with the induction step. 
Given $(M_g,A,B)$ of complexity $c$, suppose the theorem holds for any triad of smaller complexity than $c$. We apply a bad simplex argument (Theorem \ref{badsim}): 

Consider the complex $\HTD^{ +}(M_g,A,B)$, and take as the set of bad vertices of a simplex $\sigma$ those who share their outer disc with some other vertex in the simplex. This set satisfies the conditions of the bad simplex argument, and the good simplex for this criteria is precisely $\HTD(M_g,A,B)$. So it only remains to check that the good link of a bad $k$-simplex is $(m-k)$-connected.

The good link of a bad $k$-simplex $\sigma$ is isomorphic to the handle-tether-disc complex $\HTD(M_{g-k-1},A',B')$, where $M_{g-k-1}$ is the handlebody obtained by cutting off the handles in the handle-tether-discs of $\sigma$; $A'$ is the set of discs not in an outer disc of $\sigma$, plus discs corresponding to regular neighborhoods in $\partial M_{g-k-1}$ of each of the handle-tether-discs (notice that the neighborhoods can overlap, see Figure \ref{linkredgreen}); and $B'=B\cap A'$. The triad $(M_{g-k-1},A',B')$ has smaller complexity than $(M_g,A,B)$, hence by the induction hypothesis $\HTD(M_{g-k-1},A',B')$ is $m'$-connected with $$m'=\min\left\{ \left\lfloor\frac{(g-k-1)-3}{2}\right\rfloor, \left\lfloor\frac{|B'|+1}{2d-1}-2\right\rfloor,|B'|-2\right\}.$$ Following the next set of inequalities:

\begin{enumerate}
\item $\left \lfloor{\frac{(g-k-1)-3}{2}}\right \rfloor \overset{*}\geq \left \lfloor{\frac{g-2k-3}{2}}\right \rfloor\ = \left \lfloor{\frac{g-3}{2}}\right \rfloor -k$,

\item$ \left \lfloor{\frac{|B'|+1}{2d-1}-2}\right \rfloor \overset{**}{\geq} \left \lfloor{\frac{|B|-kd+1}{2d-1}-2}\right \rfloor \geq \left\lfloor{\frac{|B|+1}{2d-1}-2}\right \rfloor -\frac{kd}{2d-1}\geq  \left\lfloor{\frac{|B|+1}{2d-1}-2}\right \rfloor -k$,

\item $ \text{If    }  \left\lfloor\frac{|B'|+1}{2d-1}-2\right\rfloor\geq|B'|-2 \text{, then     } d=1 \text{ and }   |B'|-2 \overset{**}{\geq} |B|-kd-2= (|B|-2)-k$, \\ $*$: $k\geq 1$, as a single vertex cannot be bad, 
 \\$**$: $|B'|\geq |B|-kd$, as each outer disc contains $d$ discs in $B$, but at least two vertices share outer disc,

\end{enumerate}
we deduce that $m'\geq m-k$, hence the connectivity of $\HTD(M_{g-k-1},A',B')$ is at least $m-k$ for $m=\min\left\{\left\lfloor\frac{g-3}{2}\right\rfloor,\left\lfloor \frac{|B|+1}{2d-1}-2\right\rfloor, |B|-2\right\}$. 

According to Theorem \ref{badsim}, the inclusion $i$ induces an isomorphism in $\pi_l$ for every $l\leq m$. Together with Theorem \ref{teo-htb1}, this finishes the proof.

\end{proof}

\begin{figure}[H]
\includegraphics[width=13cm]{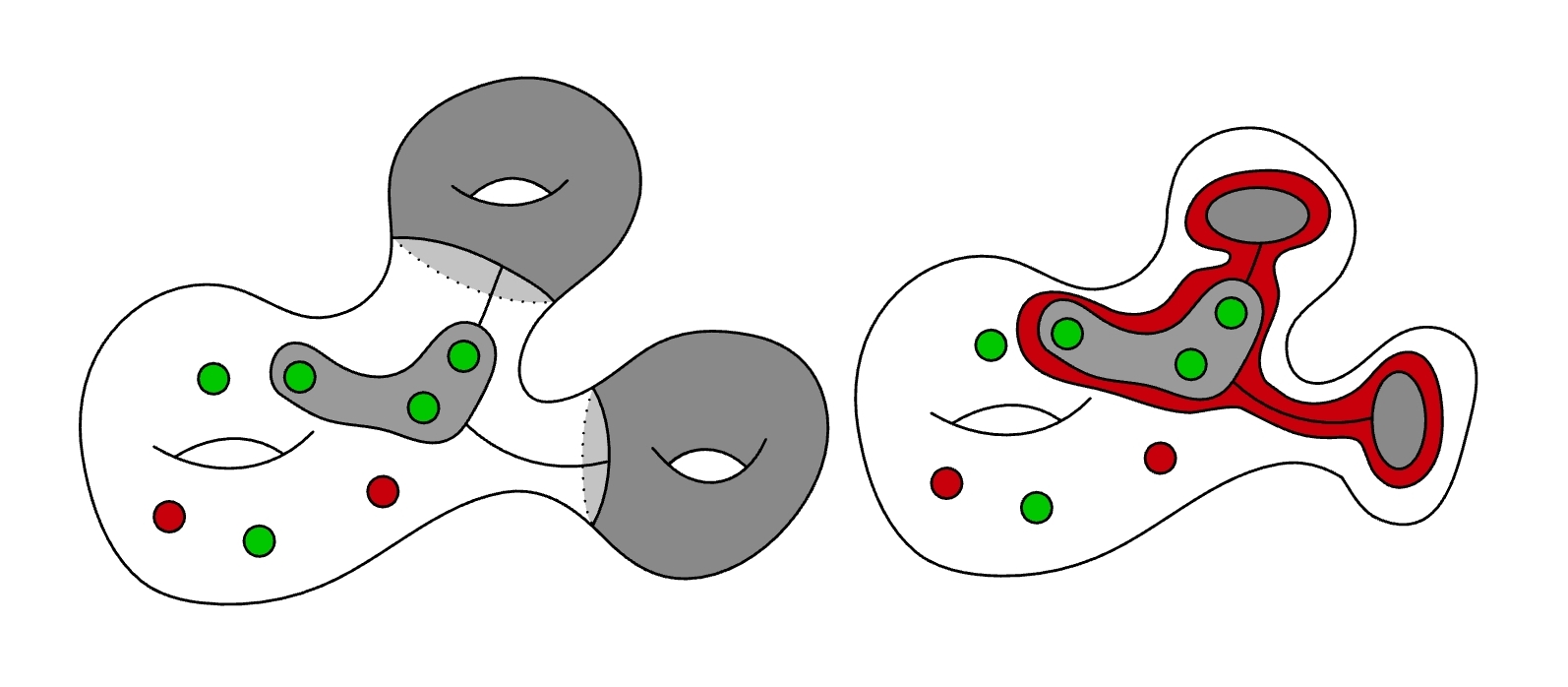}
\centering
\caption{In the left (resp. right) picture: Green spots represent spots in $B$ (resp. $B'$). Red spots represent spots in $A\setminus B$ (resp. $A'\setminus B'$).}
\label{linkredgreen}
\end{figure}

\subsection{Connectivity of the piece complex}We are ready to prove Theorem \ref{teo-piece}.

\begin{proof}[\sc Proof of Theorem {\rm \ref{teo-piece}}] 

Define a map $$p: \HTD(M_g,A,B)\rightarrow \Pc(M_g,A,B)$$ given by taking a regular neighborhood in $M_g$ of each handle-tether-disc. The map is simplicial, as disjointness is preserved. We now examine the possible choices of preimages for a given vertex. Each piece contains a unique handle up to isotopy. Similarly, there exists a single eligible outer disc up to isotopy, and the tether is also uniquely determined up to isotopy. Consequently, each piece has exactly one corresponding preimage, so the map is bijective. It follows that the map $p$ is an isomorphism, and because of Theorem \ref{teo-htb2}, the complex $\Pc(M_g,A,B)$ is $m$-connected, with $m=\min\left\{\left\lfloor\frac{g-3}{2}\right\rfloor,\left\lfloor \frac{|B|+1}{2d-1}-2\right\rfloor, |B|-2\right\}$. 

It remains to check that the link of any $k-$simplex is $m-k-1$ connected. But the link of any $k-$simplex $\sigma$ is isomorphic to $\Pc(M_{g'},A',B')$, where $M_{g'}$ is obtained from $M_g$ by cutting off a set of disjoint representatives of the pieces in $\sigma$, $A'$ are the spots of $M_{g'}$ (the ones from $A$ plus the new ones from the surgery), and $B'=B\cap A'$. The complex $\Pc(M_{g'},A',B')$ is $m'-$connected, with $m'=\min\left\{\left\lfloor\frac{g'-3}{2}\right\rfloor,\left\lfloor \frac{|B'|+1}{2d-1}-2\right\rfloor, |B'|-2\right\}$. Checking the following calculations:
\begin{enumerate}
\item $\left \lfloor{\frac{g'-3}{2}}\right \rfloor\geq\left \lfloor{\frac{(g-k-1)-3}{2}}\right \rfloor\ \geq \left \lfloor{\frac{g-2k-2-3}{2}}\right \rfloor\ \geq \left \lfloor{\frac{g-3}{2}}\right \rfloor -k-1$,

\item $ \left \lfloor{\frac{|B'|+1}{2d-1}-2}\right \rfloor \overset{*}{\geq} \left \lfloor{\frac{|B|-kd-d+1}{2d-1}-2}\right \rfloor \geq \left\lfloor{\frac{|B|+1}{2d-1}-2}\right \rfloor -\frac{kd+d}{2d-1}\geq \left\lfloor{\frac{|B|+1}{2d-1}-2}\right \rfloor -k-1$, 

\item $\text{If    }  \left\lfloor\frac{|B'|+1}{2d-1}-2\right\rfloor\geq|B'|-2 \text{, then     } d=1 \text{ and } |B'|-2 \overset{*}{\geq} |B|-kd-d-2= (|B|-2)-k-1$, \\ $*$: $|B'|\geq |B|-(k+1)d$, as each piece contains $d$ discs in $B$,

\end{enumerate}

we can see that $m'\geq m-k-1$, hence $\Pc(M_g,A,B)$ is wCM of dimension $m+1$. For $A=B$, we get the bound in the statement of the theorem.
\end{proof}

\section{Proof of Theorem I}\label{Sect3.5}

We are finally in a position to prove Theorem \ref{Main1}. Let $v$ be a vertex, which (up to the action of $\H$) we can write as $v=[(M,\id)]$, where $M$ is a spotted subhandlebody of genus $g$, with a set $A$ of spots. For $v$ of height $h(v)=k$, we have that $|A|=r+k(d-1)$ and $g=g_o+k$, where $g_o$ is the genus of $O$. Note that, since $d\geq 2$, both $|A|$ and $g$ tend to infinity as $k$ does. In particular, together with Corollary \ref{cor-lk}, this proves the following corollary, which in turn finishes the prove of Theorem \ref{Main1}.

\begin{cor}\label{TheorI}
If $d\geq 2$ and $r\geq 1$, $\dl(v,\Xc)$ satisfies Condition (4) in Theorem \ref{brown1}: for every $l\geq 1$, there exists a sufficiently large critical value $s$ such that for every vertex $v\in \Xc$ with $h(v)\geq s$, $\dl(v,\Xc)$ is $(l-1)$-connected.
\end{cor}

\section{Proof of Theorem II}\label{Sect4}

We are now in position to prove Theorem \ref{Main2}. As in the previous section, let $v=[(M,\id)]$ be a vertex, and $M$ a spotted subhandlebody of genus $g$, with a a set $A$ of spots. The genus $g=g_o+k$ tends to infinity as $k$ does, but since $d=1$, we have that $|A|=r$ $\forall k$. Together with Corollary \ref{cor-lk}, this proves the following corollary. 

\begin{cor}\label{Cond(4)}
If $d=1$ and $r\geq 1$, $\dl(v,\Xc)$ satisfies Condition (4) in Theorem \ref{brown1} for $l=r-1$: for a sufficiently large critical value $s$ and for every vertex $v\in \Xc$ with $h(v)\geq s$, $\dl(v,\Xc)$ is $(r-2)$-connected.
\end{cor}

Hence the action of  $\H_{1,r}(O,Y)$ on $\Xc_{1,r}(O,Y)$ satisfies every hypothesis of Theorem \ref{brown1} for $l=r-1$. We now check that the action satisfies Condition (5) in Theorem \ref{brown2}.

\begin{lem}\label{lem(5)}
    The action of $\H_{1,r}(O,Y)$ on $\T_{1,r}(O,Y)$ satisfies Condition (5) in Theorem \ref{brown2} for $l=r-1$: for a sufficiently large critical value $s$ and for every vertex $v\in \Xc$ with $h(v)\geq s$, $\dl(v,\Xc)$ is of dimension $r-1$.
\end{lem}

\begin{proof}[\sc Proof of Lemma {\rm \ref{lem(5)}}]
For sufficiently large $h(v)$ there are simplices in $\Pc(M_g,A)$ of dimension $r-1$, we can obtain one by selecting $r$ disjoint pieces. So $\dl(v,\Xc)=\Pi^{-1}(\Pc(M_g,A))$ also has dimension at least $r-1$ because of property (2) in Definition 
\ref{def-join}. On the other hand, $\Pc(M_g,A)$ cannot have higher dimension as the descending cell of $v$ with highest dimension has exactly dimension $r$.
\end{proof}

Before checking Condition (6) in Theorem \ref{brown2}, we need one last result, stated and proven in \cite[Lemma 4.5]{SurfaceHoughton}. A direct translation of the proof applies to our case.

\begin{lem}\label{lemainffibre}
    For every vertex $x\in\Pc(M_g,A)$, the fiber $\Pi^{-1}(x)\subset\dl(v,\Xc)$ is infinite.
\end{lem}

We now have all the necessary components to prove that Condition (6) holds in the star case:

\begin{lem}\label{lem(6)}
    The action of $\H_{1,r}(O,Y)$ on $\T_{1,r}(O,Y)$ satisfies Condition (6) in Theorem \ref{brown2}: for each critical value $s$ there exists a vertex $v\in \Xc_{1,r}(O,Y)$ with non-contractible descending link and height $h(v)\geq s$.
\end{lem}

\begin{proof}[\sc Proof of Lemma {\rm \ref{lem(6)}}]
We will use Lemma \ref{lemainffibre}. Taking any simplex of maximal dimension $\sigma=\{v_0,...,v_{r-1}\}$, each of the preimages of the vertices $v_i$ is infinite, particularly it contains two points $\{\pm v_i\}$. All of this pairs of preimages span a $(r-1)$-sphere, which is not homologically trivial as there are no $r$-simplices. So for sufficient large critical values, every descending link is non-contractible. 
\end{proof}

This completes the proof of Theorem \ref{Main2}, since every condition in both Theorem \ref{brown1} and Theorem \ref{brown2} is satisfied for $l=r-1$.

\section{Homology of the asymptotically rigid handlebody group}\label{Sect5}

This section is dedicated to studying the homology of $\H_{d,r}(O,Y)$, particularly in the case where $r=1$ and $d=2$. We focus on this case because the related Higman-Thompson group is the classical Thompson group $V=V_{2,1}$, which is known to be simple \cite{Simple} and acyclic \cite{Acyclic}, a crucial fact for our argument. This section follows the ideas in the proof of \cite[Theorem 3.1]{Fun-Kap2}, with similar arguments employed in \cite{Asymp, Ara-Fun}.

In Subsection \ref{CompacSup}, we defined $\H^c_{d,r}(O,Y)$ as the direct limit
$$\H^c_{d,r}(O,Y):=\underset{\longrightarrow} \lim \, \H(O_k),$$ were the $O_k$ are the manifolds involved in the construction of $\T_{d,r}(O,Y)$. Since the inclusion property holds for handlebodies (see Appendix \ref{Ap1}), the embedding $O_{k_1}\hookrightarrow O_{k_2}$ induces a monomorphism $j:\H(O_{k_1})\hookrightarrow \H(O_{k_2})$, which in turn induces a homomorphism on homology $j_{i*}:H_i(\H(O_{k_1}))\rightarrow H_i(\H(O_{k_2}))$ for each $i$. A key result by Hatcher and Wahl \cite[Theorem 1.8]{Disk} establishes the homological stabilization of handlebody groups. We restate this result in a form suited to our context below.

\begin{theor}
    Given any handlebody $O$, and a genus-one handlebody $Y$, if the genus of $O_{k}$ is greater or equal than $2i+4$, then
    $$H_i(\H(O_{k}))\cong H_i(\H(M_g)),$$
    where $M_g$ is any handlebody of genus $g\geq 2i+4$.
\end{theor}

As $\H^c_{2,1}(O,Y)$ is the direct limit of the groups $\H(O_k)$, and the maps in homology stabilize, we get:

\begin{cor} \label{CorSt}
   Let $O$ be any handlebody, and let $Y$ be a genus-one handlebody. Let $M_g$ be any handlebody of genus $g\geq 2i+4$. Then
   $$H_i(\H^c_{2,1}(O,Y)) \cong H_i(\H(M_g)).$$
\end{cor}

In particular, since each $\H(M_g)$ is of type $F_\infty$ (by \cite[Theorem 1.1 a, Theorem 6.1]{McCullough}) every homology group $H_i(\H(M_g))$ is finitely generated. Thus we obtain:

\begin{cor}\label{CorFin}
   The group $H_i(\H^c_{2,1}(O,Y))$ is finitely generated for every $i\geq 0$.
\end{cor}

The connection between the homological stability results for handlebody groups, and the homology of the asymptotically rigid handlebody becomes apparent in the following result. Its proof closely follows the proof of \cite[Proposition 10.5]{Asymp}.

\begin{theor} \label{TeoHomo}
   Let $O$ be any handlebody, and let $Y$ be a genus-one handlebody. For every $i\geq 0$, $$H_i(\H_{2,1}(O,Y))\cong H_i(\H^c_{2,1}(O,Y)).$$

\end{theor}

\begin{proof}[\sc Proof of Theorem {\rm \ref{TeoHomo}}]
As we saw in Subsection \ref{SubRelThom}, there is a short exact sequence:  

$$1\rightarrow\H^c_{2,1}(O,Y)\xrightarrow{}\H_{2,1}(O,Y)\xrightarrow{} V\rightarrow 1.$$

We apply the Lyndon-Hochschild-Serre spectral sequence, getting

$$E^2_{p,q}=H_p(V,H_q(\H_{2,1}^c(O,Y),\ZZ))\Rightarrow H_i(\H_{2,1}(O,Y),\ZZ).$$

For convenience, we write $A_q=H_q(\H_{2,1}^c(O,Y),\ZZ)$. Because of Corollary \ref{CorFin}, we know $A_q$ is finitely generated and abelian, particularly it is residually finite. By \cite[Theorem 1]{Baumslag}, $\Aut(A_q)$ is residually finite. On the other hand, the group $V$ is simple \cite{Simple}. These two fact combined imply that any homomorphism $V\rightarrow A_q$ is trivial, and consequently, any action of $V$ in $A_q$ must be trivial. Since $V$ is acyclic \cite{Acyclic}, every $H_p(V,A_q)$ for $p\geq 1$ must be trivial and $H_0(V,A_q)\cong A_q$. Thus the spectral sequence collapses, and $H_i(\H_{2,1}(O,Y),\ZZ)\cong H_i(\H_{2,1}^c(O,Y),\ZZ)$ for all $i$.

\end{proof}

A combination of Theorem \ref{TeoHomo} and Corollary \ref{CorSt} yields Corollary \ref{Corst2}, which states what is the homology of the asymptotically rigid handlebody group.

\appendix
\section{Inclusion, intersection, and cancellation properties}\label{Ap1}

In this appendix we introduce the inclusion, intersection, and cancellation properties, and prove that tree handlebodies satisfy them. We will deduce these properties from the fact that they are satisfied for surfaces, which is established in \cite[Apendix B1]{Asymp}.

\begin{defi}(Inclusion property). We say that $\T_{d,r}(O,Y)$ has the \textit{inclusion property} if the following holds: Let $M$ be a suited handlebody, and $N\subset M$ a connected spotted subhandlebody that is either suited or the complement of a suited handlebody with at least two spots. Then the homomorphism $\H(N)\rightarrow\H(M)$ induced by the inclusion map $N\hookrightarrow M$ is injective.
\end{defi}

\begin{lem} \label{inclusion}
    For all $d,r\geq 1$, $\T_{d,r}(O,Y)$ satisfies the inclusion property.
\end{lem}

\begin{proof}[\sc Proof of Lemma {\rm \ref{inclusion}}]
Recall the map $\R$ in Subsection \ref{relation}. Consider the following diagram,

\[
\begin{tikzcd}[row sep=0.4cm,column sep=0.4cm]
\H(N)\arrow[d,"\R"] \arrow[hookrightarrow]{r}{i}& \H(M)\arrow[d,"\R"] \\
\Map(S_{N}) \arrow[hookrightarrow]{r}{i}& \Map(S_{M})
\end{tikzcd}
\]

The diagram commutes, as given $f\in\H(N)$ and a representative $f_1$ of $f$, restricting $f_1$ to $S_{N}$ and then extending it by the identity to $S_{M}$, yields the same homeomorphism as extending $f_1$ to $M$ and then restricting it to $S_{M}$. Let $g$ and $f$ be distinct elements of $\H(N)$. Restriction to the boundary is injective because of \cite[Lemma 3.1]{Hensel} so $\R(g)\neq \R(f)$. Because of the inclusion property for surfaces (see \cite[Section 3.6]{MCG}), $i(\R(g))\neq i(\R(f))$. Since the diagram commutes, it follows that $\R(i(g))\neq \R(i(f))$. Restriction to the boundary is well defined, so $i(g)\neq i(f)$. 

\end{proof}

\begin{defi}(Intersection property). Let $\T_{d,r}(O,Y)$ be a tree handlebody with the inclusion property. We say that $\T_{d,r}(O,Y)$ has the \textit{intersection property} if, for every suited handlebody $M$, the following holds: Let $L_1$ and $L_2$ be disjoint spotted subhandlebodies, each homeomorphic to $Y^d$, such that all but one of the spots in $L_i$ are spots in $M$. Write $M_i$ for the complement of $L_i$ in $M$, and set $N=M_1\cap M_2$. Then $\H(N)=\H(M_1)\cap \H(M_2)$.
\end{defi}

\begin{lem}\label{leminter}
    Let $N\subset M$ be two spotted subhandlebodies of $\T_{d,r}(O,Y)$. Then, $\H(N)=\Map(S_{N})\cap\H(M)$ as subgroups of $\Map(S_{M})$.
\end{lem}

\begin{proof}[\sc Proof of Lemma {\rm \ref{leminter}}]

The inclusion from left to right is immediate. For the other inclusion, let $f\in\Map(S_{N})\cap\H(M)\subset\Map(S_{M})$. As $f\in\Map(S_{N})$, we can take $f_1$ a representative fixing $S_{M}\setminus S_{N}$. Because $f\in\H(M)$, we can take $f_2$ a representative which extends to a self-homeomorphism $g_2$ of $M$.

Select a regular neighborhood $R_M$ of $S_{M}$ in $M$, and parametrize $R_M$ as $S_{M}\times[0,1]$ (with $S_{M}\times\{1\}$ being $S_{M}$ itself). Take an isotopy $\phi_t$ from $f_2$ to $f_1$, such that $\phi_0=f_2$ and $\phi_1=f_1$. We use $\phi_t$ to define a map $\Phi(p,x):=(\phi_{x}(p),x)$ in $S_{M}\times[0,1]$, which agrees with $f_1$ in $S_{M}\times\{1\}= S_M$, and with $f_2$ in $S_{M}\times\{0\}= \partial \overline{M\setminus R_M}$. 

The handlebody $\overline{M\setminus R_M}$ is homeomorphic to $M$. Let $h:\overline{M\setminus R_M}\rightarrow M$ be a homeomorphism. We define a homeomorphism $$g_3:=  \begin{cases}
       h^{-1} \circ g_2\circ h &\quad\text{in}\quad \overline{M\setminus R_M},\\
       \Phi(p,x) &\quad\text{in}\quad R_M,\\
     \end{cases}$$

which fixes $S_{M}\setminus S_{N}$, although not necessarily $M\setminus N$. However, it sends the spots of $N$ to discs with the same curve as boundary, and according to Lemma 2.3 in \cite{Hensel}, any two discs sharing their boundaries are isotopic. Composing the isotopies between discs with $g_3$, we can construct a homeomorphism whose restriction to $S_M$  is $f_1$, and that fixes $M\setminus N$ (it fixes $\partial (M\setminus N)$ so it also fixes the interior, because of \cite[Lemma 3.1]{Hensel}), hence $f\in \H(N)$.

\end{proof}

\begin{lem}\label{intersection}
    For all $d,r\geq 1$, $\T_{d,r}(O,Y)$ satisfies the intersection property.
\end{lem}

\begin{proof}[\sc Proof of Lemma {\rm \ref{intersection}}]

Take the equality $\Map(S_{N})=\Map(S_{M_1})\cap\Map(S_{M_2})$, which is true because of the intersection property for Cantor surfaces (a direct consequence of the Alexander method \cite[Section 2.3]{MCG}). Set $M=M_1\cup M_2$. Intersection at both sides of the equality $\Map(S_{N})=\Map(S_{M_1})\cap\Map(S_{M_2})$ by $\H(M)$ and use of Lemma \ref{leminter}, yields the equality $\H(N)=\H(M_1)\cap\H(M_2)$.

\end{proof}

\begin{defi} (Cancellation property). We say that $\T_{d,r}(O,Y)$ satisfies the cancellation property if the following holds: Let $M$ be a suited handlebody with at least one piece, and $D$ a separating disc that cuts off a spotted subhandlebody homeomorphic to $Y^d$. Then the remaining component $N\subset M$ is homeomorphic to some suited handlebody of $\T_{d,r}(O,Y)$.    
\end{defi}

\begin{lem}\label{cancellation}
     For all $d,r\geq 1$, $\T_{d,r}(O,Y)$ satisfies the cancellation property.
\end{lem}

\begin{proof}[\sc Proof of Lemma {\rm \ref{cancellation}}]

There is exactly one handlebody up to homeomorphism for each genus (see \cite[Proposition 9.2.19]{martelli}). If there is a homeomorphism between two spotted handlebodies, we can always take such homeomorphism to take spots to spots as long as the number of them agree. Choose any suited handlebody $N'$ of the same genus as $N$ and containing $O_0$, we can find a homeomorphism taking $N'$ to $N$.

\end{proof}

\bibliographystyle{plain}
\bibliography{name.bib}

\begin{thebibliography}{10}

\bibitem{Asymp}
Javier Aramayona, Kai-Uwe Bux, Jonas Flechsig, Nansen Petrosyan, and Xiaolei Wu.
\newblock Asymptotic mapping class groups of {C}antor manifolds and their finiteness properties.
\newblock {\em Rev. Mat. Iberoam.}, 40(6):2003--2072, 2024.
\newblock With an appendix by Oscar Randal-Williams.

\bibitem{SurfaceHoughton}
Javier Aramayona, Kai-Uwe Bux, Heejoung Kim, and Christopher~J. Leininger.
\newblock Surface {H}oughton groups.
\newblock {\em Math. Ann.}, 389(4):4301--4318, 2024.

\bibitem{Ara-Fun}
Javier Aramayona and Louis Funar.
\newblock Asymptotic mapping class groups of closed surfaces punctured along {C}antor sets.
\newblock {\em Mosc. Math. J.}, 21(1):1--29, 2021.

\bibitem{Baumslag}
Gilbert Baumslag.
\newblock Automorphism groups of residually finite groups.
\newblock {\em Journal of the London Mathematical Society}, s1-38(1):117--118, 01 1963.

\bibitem{MorseTh}
Mladen Bestvina and Noel Brady.
\newblock Morse theory and finiteness properties of groups.
\newblock {\em Invent. Math.}, 129(3):445--470, 1997.

\bibitem{Brown}
Kenneth~S. Brown.
\newblock Finiteness properties of groups.
\newblock In {\em Proceedings of the {N}orthwestern conference on cohomology of groups ({E}vanston, {I}ll., 1985)}, volume~44, pages 45--75, 1987.

\bibitem{MCG}
Benson Farb and Dan Margalit.
\newblock {\em A primer on mapping class groups}, volume~49 of {\em Princeton Mathematical Series}.
\newblock Princeton University Press, Princeton, NJ, 2012.

\bibitem{Fun-Kap1}
L.~Funar and C.~Kapoudjian.
\newblock On a universal mapping class group of genus zero.
\newblock {\em Geom. Funct. Anal.}, 14(5):965--1012, 2004.

\bibitem{Funar}
Louis Funar.
\newblock Braided {H}oughton groups as mapping class groups.
\newblock {\em An. \c Stiin\c t. Univ. Al. I. Cuza Ia\c si. Mat. (N.S.)}, 53(2):229--240, 2007.

\bibitem{Fun-Kap2}
Louis Funar and Christophe Kapoudjian.
\newblock An infinite genus mapping class group and stable cohomology.
\newblock {\em Comm. Math. Phys.}, 287(3):784--804, 2009.

\bibitem{BrThom}
Anthony Genevois, Anne Lonjou, and Christian Urech.
\newblock Asymptotically rigid mapping class groups, {I}: {F}initeness properties of braided {T}hompson's and {H}oughton's groups.
\newblock {\em Geom. Topol.}, 26(3):1385--1434, 2022.

\bibitem{Geoghegan}
Ross Geoghegan.
\newblock {\em Topological methods in group theory}, volume 243 of {\em Graduate Texts in Mathematics}.
\newblock Springer, New York, 2008.

\bibitem{Tethers}
Allen Hatcher and Karen Vogtmann.
\newblock Tethers and homology stability for surfaces.
\newblock {\em Algebr. Geom. Topol.}, 17(3):1871--1916, 2017.

\bibitem{Disk}
Allen Hatcher and Nathalie Wahl.
\newblock Stabilization for mapping class groups of 3-manifolds.
\newblock {\em Duke Math. J.}, 155(2):205--269, 2010.

\bibitem{Hensel}
Sebastian Hensel.
\newblock A primer on handlebody groups.
\newblock In {\em Handbook of group actions. {V}}, volume~48 of {\em Adv. Lect. Math. (ALM)}, pages 143--177. Int. Press, Somerville, MA, [2020] \copyright 2020.

\bibitem{Unicorn}
Sebastian Hensel, Piotr Przytycki, and Richard C.~H. Webb.
\newblock 1-slim triangles and uniform hyperbolicity for arc graphs and curve graphs.
\newblock {\em J. Eur. Math. Soc. (JEMS)}, 17(4):755--762, 2015.

\bibitem{Higman}
Graham Higman.
\newblock {\em Finitely presented infinite simple groups}, volume No. 8 of {\em Notes on Pure Mathematics}.
\newblock Australian National University, Department of Pure Mathematics, Department of Mathematics, I.A.S., Canberra, 1974.

\bibitem{Simple}
Graham Higman.
\newblock {\em Finitely presented infinite simple groups}, volume No. 8 of {\em Notes on Pure Mathematics}.
\newblock Australian National University, Department of Pure Mathematics, Department of Mathematics, I.A.S., Canberra, 1974.

\bibitem{martelli}
Bruno Martelli.
\newblock {\em An Introduction to Geometric Topology}.
\newblock Independently published, 2023.

\bibitem{McCullough}
Darryl McCullough.
\newblock Virtually geometrically finite mapping class groups of {$3$}-manifolds.
\newblock {\em J. Differential Geom.}, 33(1):1--65, 1991.

\bibitem{Quill}
Daniel Quillen.
\newblock Homotopy properties of the poset of nontrivial {$p$}-subgroups of a group.
\newblock {\em Adv. in Math.}, 28(2):101--128, 1978.

\bibitem{Braided}
Rachel Skipper and Xiaolei Wu.
\newblock Finiteness properties for relatives of braided {H}igman--{T}hompson groups.
\newblock {\em Groups Geom. Dyn.}, 17(4):1357--1391, 2023.

\bibitem{Aprox}
Edwin~H. Spanier.
\newblock {\em Algebraic topology}.
\newblock Springer-Verlag, New York, [1995?].
\newblock Corrected reprint of the 1966 original.

\bibitem{Acyclic}
Markus Szymik and Nathalie Wahl.
\newblock The homology of the {H}igman-{T}hompson groups.
\newblock {\em Invent. Math.}, 216(2):445--518, 2019.

\end{thebibliography}

\end{document}